%% file: main.tex
\documentclass[11pt, reqno]{amsart}
\usepackage{amsmath,amsthm,amssymb,enumerate}
\usepackage[all]{xy}
\usepackage{array}

\usepackage[dvipsnames]{xcolor}
\definecolor{citation}{rgb}{0,.40,.80}
\usepackage[colorlinks,linkcolor=citation,citecolor=citation,urlcolor=citation]{hyperref}

\newcommand \C {\mathbb C}
\newcommand \cC {\mathcal C}

\newcommand \fF {\mathcal F}
\newcommand \gG {\mathcal G}
\newcommand \hH {\mathcal H}

\renewcommand \L {\mathcal L}

\renewcommand \O {\mathcal O}
\renewcommand \P {\mathbb P}
\newcommand \Q {\mathbb Q}
\newcommand \Ql {\mathbb Q_{\ell}}
\newcommand \R {\mathbb R}

\newcommand \Z {\mathbb Z}
\newcommand\et{\textrm{\'et}}
\newcommand\xra{\xrightarrow}

\DeclareMathOperator \alb {alb}
\DeclareMathOperator \Aut {Aut}

\DeclareMathOperator \hyp {hyp}
\DeclareMathOperator \CH {CH}
\DeclareMathOperator \charr {char}
\DeclareMathOperator \el {ell}
\DeclareMathOperator \id {id}

\DeclareMathOperator \Hilb {Hilb}
\DeclareMathOperator \Kt {K3}
\DeclareMathOperator \OO {O}

\DeclareMathOperator \Pic {Pic}
\DeclareMathOperator \NS {NS}
\newcommand \Picbar {\overline{\mathrm{Pic}}{}} 
\DeclareMathOperator \Spec {Spec}

\DeclareMathOperator \Kum {Kum}
\DeclareMathOperator{\Fix}{Fix}
\DeclareMathOperator{\Gal}{Gal}
\DeclareMathOperator{\Id}{Id}
\DeclareMathOperator{\supp}{supp}
\DeclareMathOperator{\sym}{sym}
\DeclareMathOperator{\Sym}{Sym}
\DeclareMathOperator{\im}{im}
\DeclareMathOperator{\sep}{sep}
\newcommand \kbar {{\bar{k}}}
\DeclareMathOperator{\Char}{char}
\DeclareMathOperator{\Span}{Span}
\DeclareMathOperator{\Frac}{Frac}

\DeclareMathOperator{\shF}{\mathcal{F}}

\newtheorem {thm} {Theorem}[section]
\newtheorem {cor} [thm] {Corollary}

\newtheorem {lem} [thm] {Lemma}
\newtheorem {lemma} [thm] {Lemma}
\newtheorem {prop} [thm] {Proposition}
\newtheorem {claim} [thm] {Claim}

\theoremstyle{definition}
\newtheorem {defn} [thm] {Definition}
\newtheorem {setting} [thm] {Setting}
\newtheorem {cons} [thm] {Construction}
\newtheorem {question} {Question}

\newtheorem {example}[thm] {Example}

\newtheorem {rmk}[thm] {Remark}

\numberwithin{equation}{section}

\makeatletter
\newcommand{\xleftrightarrow}[2][]{\ext@arrow 3359\leftrightarrowfill@{#1}{#2}}
\newcommand{\dashto}[2][]{\ext@arrow 0359\rightarrowfill@@{#1}{#2}}
\newcommand{\xdashleftarrow}[2][]{\ext@arrow 3095\leftarrowfill@@{#1}{#2}}
\newcommand{\xdashleftrightarrow}[2][]{\ext@arrow 3359\leftrightarrowfill@@{#1}{#2}}
\def\rightarrowfill@@{\arrowfill@@\relax\relbar\rightarrow}
\def\leftarrowfill@@{\arrowfill@@\leftarrow\relbar\relax}
\def\leftrightarrowfill@@{\arrowfill@@\leftarrow\relbar\rightarrow}
\def\arrowfill@@#1#2#3#4{%
  $\m@th\thickmuskip0mu\medmuskip\thickmuskip\thinmuskip\thickmuskip
   \relax#4#1
   \xleaders\hbox{$#4#2$}\hfill
   #3$%
}
\makeatother

\begin{document}

\title[Groups of symplectic involutions]{Groups of symplectic involutions \\ on symplectic  varieties of Kummer type \\ and their fixed loci}

\author{Sarah Frei and Katrina Honigs}

\date{}

\begin{abstract}
We describe the Galois action on the middle $\ell$-adic cohomology of smooth, projective fourfolds $K_A(v)$ that occur as a fiber of the Albanese morphism on moduli spaces of sheaves on an abelian surface $A$ with Mukai vector $v$. We show this action is determined by the action on $H^2_{\et}(A_{\bar{k}},\Q_\ell(1))$ and on a subgroup $G_A(v) \leqslant (A\times \hat{A})[3]$, which depends on $v$. This generalizes the analysis carried out by Hassett and Tschinkel over $\C$ \cite{HasTsc}. As a consequence, over number fields, we give a condition under which $K_2(A)$ and $K_2(\hat{A})$ are not derived equivalent.

The points of $G_A(v)$ correspond to involutions of $K_A(v)$. Over $\C$, they are known to be symplectic and contained in the kernel of the map $\Aut(K_A(v))\to \mathrm{O}(H^2(K_A(v),\Z))$. We describe this kernel for all varieties $K_A(v)$ of dimension at least $4$. 

When $K_A(v)$ is a fourfold over a field of characteristic~0, the fixed-point loci of the involutions 
contain K3 surfaces whose cycle classes span a large portion of the middle cohomology. We examine the fixed loci in fourfolds $K_A(0,l,s)$ over $\C$ where $l$ is a $(1,3)$-polarization, finding the K3 surface
to be
elliptically fibered under a Lagrangian fibration of $K_A(0,l,s)$.
\end{abstract}

\maketitle

\section{Introduction}

\input{introduction}

\section{Moduli spaces over arbitrary fields}
\label{sec:backgroundmoduli}
\input{backgroundmoduli}

\section{Symplectic involutions on $K_A(v)$}
\label{sec:involutions}

\input{involutions}

\section{The middle cohomology of fourfolds $K_A(v)$}
\label{sec:midcohomology}

\input{cohomology}

\section{Relation to derived equivalences}
\label{sec:derived}

\input{derivedcategories}

\section{A $(1,3)$-polarized example: Lagrangian fibrations}
\label{sec:JacobiansK3}

\input{JacobiansK3}

\section{A $(1,3)$-polarized example: Singular fibers of an elliptic K3}
\label{sec:Jacobianssingular}

\input{Jacobianssingular}

\section{A $(1,3)$-polarized example: Isolated points}
\label{sec:JacobiansIsolatedpts}

\input{JacobiansIsolatedpts}

%%------{BIBLIOGRAPHY}-----------------------------------------------------------------------------------
\bibliographystyle{alpha}
\bibliography{mainbib}

\medskip

{\footnotesize
\noindent Sarah Frei \\
Department of Mathematics \\
Dartmouth College\\ 
27 N. Main Street\\ 
Hanover, NH 03755\\
sarah.frei@dartmouth.edu \bigskip

\noindent Katrina Honigs \\
Department of Mathematics\\
Simon Fraser University\\
8888 University Drive\\
Burnaby, B.C.\\
Canada V5A 1S6\\
khonigs@sfu.ca
}

\end{document}

%% file: introduction.tex
Given a polarized abelian surface $(A, H)$ defined over an arbitrary field $k$, we may study moduli spaces of geometrically $H$-stable sheaves on $A$ with a fixed Mukai vector $v=(r,l,s)$, that is, fixed rank, N\'eron-Severi class of the determinental line bundle, and Euler characteristic. Under mild conditions on the Mukai vector, the moduli spaces  $M_A(v)$ are smooth and projective. Their Albanese varieties are  $A\times\hat{A}$, and we denote 
a fiber
of the Albanese morphism by $K_A(v)$.

If defined over $\C$, the variety $K_A(v)$ is a
hyperk\"ahler variety of dimension $v^2-2$ and is deformation equivalent to the \textit{generalized Kummer variety} $K_n(A)\cong K_A(1,0,-n-1)$ where $n:=\frac{v^2}{2}-1$, which is
given by the fiber over $0$ of the summation map acting on the Hilbert scheme of length-$(n+1)$ points on $A$. 
Following Fu and Li \cite{FuLi}, who study these varieties over other fields, we call the $K_A(v)$ symplectic varieties (see Proposition~\ref{prop:Kvsmooth}). 
There are four known deformation types of hyperk\"ahler varieties: $\Kt^{[n]}$-type, Kummer type (or $\Kum_n$-type), and the two sporadic examples of O'Grady \cite{OG10,OG6}.
The varieties $K_A(v)$ are of Kummer $n$-type. 
It has been shown
\cite[Prop.~2.4]{MW} that under a lattice-theoretic condition, if $n+1$ is a prime power, any hyperk\"ahler of 
Kummer $n$-type is the fiber of the Albanese map of a
moduli space of stable objects on an abelian surface $A$.
So varieties $K_A(v)$
do not exhaust the class, but are at this point the best understood.

In \cite{HasTsc}, Hassett and Tschinkel analyze the cohomology of complex generalized Kummer fourfolds $K_2(A)$. They show that $H^4(K_2(A), \Q)$ is generated by $H^2(K_2(A),\Q)$ and an $81$-dimensional vector space spanned by the cycle classes of $81$ distinct K3 surfaces in $K_2(A)$. These surfaces are each
contained in the fixed locus of a symplectic involution of the form $t_x^*\iota^*$ where $\iota$ is multiplication by $-1$ on $A$ and $t_x$ is translation by 
a point of the three-torsion $A[3]$ of $A$.
Hassett and Tschinkel use deformation theory to show that 
the middle cohomology for any hyperk\"ahler variety $X$ of $\Kum_2$-type
has a similar decomposition.
The cohomology of Kummer-type hyperk\"ahler varieties is also studied in \cite{LLVdecomp}.

In this paper, we
extend these results
by characterizing the Galois action on the $\ell$-adic \'etale cohomology of
fourfolds $K_A(v)$
over non-closed fields
As one might expect from  the results of  Hassett--Tschinkel,
there is an  $81$-dimensional subspace of
$H^4_{\et}(K_2(A)_{\bar{k}},\Ql(2))$ whose Galois action is determined by the structure of $A[3]$. However, deformation-theoretic tools are too coarse to keep track of
how the Galois action changes for other fourfolds $K_A(v)$,
which we find depends on $v$:

\begin{thm}[Theorem~\ref{thm:middlecohom}, Proposition~\ref{charp}]
  \label{thm:introcohom}
Suppose $K_A(v)$ is a smooth, projective variety over an arbitrary field $k$.
Then there is a subgroup $G_{A_{\kbar}}(v) \leqslant (A_{\kbar}\times \hat{A}_{\kbar})[3]$ and a Galois equivariant isomorphism  
$$H^4_{\et}(K_A(v)_{\bar{k}},\Q_\ell(2))\cong \Sym^2H^2_{\et}(K_A(v)_{\bar{k}},\Q_\ell(1)) \oplus V,$$
where $V$ is the $80$-dimensional subrepresentation of the permutation representation $\Q_\ell[G_{A_{\bar{k}}}(v)]$ such that
\[\Q_\ell[G_{A_{\bar{k}}}(v)]\cong V \oplus \Q_\ell,\]
and the trivial representation $\Q_\ell$ is the span of $(0,0)\in G_{A_{\bar{k}}}(v)$. The Galois action on the group ring $\Q_\ell[G_{A_{\bar{k}}}(v)]$ is induced by the action on $G_{A_{\bar{k}}}(v)$.
\end{thm}

By a generalization of work of Yoshioka \cite{Yoshioka}, this means the Galois action on the middle cohomology is determined by the action on $H^2_{\et}(A_{\bar{k}},\Q_\ell(1))$ and the action on the subgroup $G_{A_{\bar{k}}}(v)$, which is the kernel of the isogeny $\phi\colon A\times\hat{A} \to A\times\hat{A}$ given by $(x,y)\mapsto (\phi_M(y)-sx,\phi_L(x)+ry)$ (See Section~\ref{sec:automorphisms}).
This stands in surprising contrast to the case of moduli spaces of sheaves on K3 surfaces---symplectic varieties of K3$^{[n]}$-type---where the cohomology representations depend only on that of the K3 surface \cite[Thm.~2]{Frei}. 

As a consequence, by studying the even cohomology of $K_2(A)$ for $A$ defined over a number field,
we are able to
show the following result on derived equivalence:

\begin{cor}[Corollary~\ref{cor:notderived}]\label{cor:intronotde}
Let $A$ be an abelian surface over a number field $k$ for which the permutation representations associated to $A_{\bar{k}}[3]$ and $\hat{A}_{\bar{k}}[3]$ are not isomorphic. Then $K_2(A)$ and $K_2(\hat{A})$ are not derived equivalent over $k$.
\end{cor}
%\begin{cor}[Corollary~\ref{cor:notderived}]\label{cor:intronotde}
%For an abelian surface $A$ over a number field $k$, the generalized Kummer fourfolds $K_2(A)$ and $K_2(\hat{A})$ are not in general derived equivalent. That is, there is an abelian surface $A$ over a number field $k$ so that 
%$K_2(A)$ and $K_2(\hat{A})$ are not derived equivalent over $k$.
%\end{cor}

In forthcoming work \cite{FHV} on Galois actions on torsion subgroups of abelian surfaces, examples of such abelian surfaces are constructed.
Intriguingly, this corollary shows that if $K_2(A)$ and $K_2(\hat{A})$ are derived equivalent after base change to $\C$, then the kernel of the Fourier--Mukai transform cannot be given by naturally associated bundles that would descend to the field of definition for $A$.
Corollary \ref{cor:intronotde} complements the recent work of
Magni \cite{Magni}, which provides a 
sufficient condition for the existence of such equivalences over algebraically closed fields of characteristic zero.

\medskip

The cohomology group $V$ in Theorem~\ref{thm:introcohom} is generated by K3 surfaces contained in the fixed-point loci of symplectic involutions on $K_A(v)$.
We give a case-by-case
explicit description of $G_A(v)$, and hence 
an explicit description of these symplectic involutions, which dictate the Galois action on $V$.

By work of Boissi\`ere--Nieper-Wisskirchen--Sarti in \cite{HigherEn}, Hassett--Tschinkel in \cite{HasTsc},
and Kapfer--Menet in \cite{KapferMenet},
for any hyperk\"ahler variety $X$ over $\C$ of $\Kum_{n-1}$-type,
the kernel
\[
\ker(\Aut(X)\to \OO(H^2(X,\Z)))\cong \Z/2\Z \ltimes (\Z/n\Z)^4
\]
consists of symplectic automorphisms of $X$; when $\dim X=4$, the kernel contains all of the symplectic involutions of $X$.
We give an explicit description of this
kernel for hyperk\"ahler varieties $K_A(v)$ of any dimension at least~$4$ over $\C$: 

\begin{thm}[Theorem~\ref{thm:symplectic}]\label{thm:introsymplectic}
  Suppose $K_A(v)$ is a smooth, projective variety over $k=\C$.
Then
\[\ker(\Aut(K_A(v))\to \OO(H^2(K_A(v),\Z)))\] 
consists of automorphisms
of the
following two forms:
$$L_y\otimes t_x^* \quad\text{and}\quad 
\kappa_{(x,y)}:=L_y\otimes t_x^*\kappa,$$
where $\kappa=\iota^*$ if $K_A(v)$ is an Albanese fiber over symmetric line bundles, and
otherwise $\kappa$ is a
composition of $\iota^*$ with a translation.
The $\kappa_{(x,y)}$ are symplectic involutions of $K_A(v)$, and when $\dim K_A(v)=4$, these are all of the symplectic involutions. 
\end{thm}

In the complex case, the group $G_{A}(v)$ also appears in \cite{MaKummers} as $\Gamma_v$.
Markman defines $\Gamma_v$ as the kernel of the map $\phi$ above as well
as in terms of Clifford algebras (\S10.1, Remark~4.3 \textit{op.cit.}).
The result \cite[Lemma~10.1]{MaKummers} and its proof
shows $\Gamma_v$ embeds into the monodromy group of $K_A(v)$,
acts trivially on $H^2(K_A(v),\Z)$ and $H^3(K_A(v),\Z)$, and 
that $M_A(v)$  is isomorphic to a quotient of $A\times\hat{A}\times K_A(v)$ by an action of $\Gamma_v$. Thus the fact that the automorphisms $L_y\otimes t_x^*$ are symplectic is not new, but we provide a proof to make our study of this family self-contained.

\medskip

Beyond their analysis of the middle cohomology for $K_2(A)$, Hassett and Tschinkel explicitly describe the fixed-point loci of the symplectic involutions. They show that the locus fixed by the standard involution contains the Kummer K3 surface
\[\overline{\{(a_1,a_2,a_3)\mid a_1=0, a_2=-a_3, a_2\neq 0\}},\]
as well as a unique isolated point
supported at the identity element $0$.
Tar\'i in \cite{Tari}
finishes the description by showing
there are 35 more isolated points, which are tuples of two-torsion points of $A$.
The deformation invariance of the symplectic involutions implies that the fixed locus of any $\iota_{(x,y)}$ in $K_A(v)$ also consists of a K3 surface and 36 isolated points \cite[Thm.~7.5]{KapferMenet}.

Motivated by these results, we seek a similar description of the fixed-point loci in 
fourfolds $K_A(0,l,s)$, whose general member is a degree $s+3$ line bundle on a genus 4 curve in the linear system $|L|$ with $c_1(L)=l$. These moduli spaces admit a Lagrangian fibration, which aids in our study. We give the following description:

\begin{thm}[Theorem~\ref{thm:badfibers}]
\label{thm:introfibration}
The K3 surface in the fixed-point locus of $\iota^*$ acting on $K_A(0,l,s)$ is elliptically fibered with four singular fibers of type $I_1$ and $10$ singular fibers of type $I_2$.
\end{thm}

The singular fibers in this elliptic fibration agree with a natural elliptic fibration on the Kummer K3 surface $K_1(A)$ when $A$ is $(1,3)$-polarized---a necessary condition for $K_A(0,l,s)$ to be a fourfold. The K3 surface appears to be closely connected to the relative Jacobian of $K_1(A)\to \P^1$. 

We also describe the isolated points in the fixed-point locus using the Abel map for the curves in $|L|$.

\subsubsection*{Outline}

In Section~\ref{sec:backgroundmoduli}, we provide a brief introduction to moduli spaces of sheaves, and Kummer-type varieties arising from them, over arbitrary fields. In Section~\ref{sec:involutions}, we identify which automorphisms of $M_A(v)$ given by translation and tensoring by a degree 0 line bundle restrict to automorphisms of $K_A(v)$, and then show how these give rise to the description of the symplectic automorphisms discussed in Theorem~\ref{thm:introsymplectic}. We also begin the analysis of the fixed-point loci for the symplectic involutions. In Section~\ref{sec:midcohomology}, we study the middle cohomology of fourfolds $K_A(v)$, proving Theorem~\ref{thm:introcohom}. In Section~\ref{sec:derived}, we compare our results to questions about derived equivalences between abelian surfaces and their generalized Kummer fourfolds. Namely, we give criteria in Section~\ref{sec:r} for when a derived equivalence between abelian surfaces $A$ and $B$ induces an isomorphism between $G_A(v)$ and $G_B(v)$, and we prove Corollary~\ref{cor:intronotde} in Section~\ref{sec:de}.

The second half of the paper is dedicated to studying the fixed-point locus of $\iota^*$ for fourfolds $K_A(0,l,s)$ over $\C$, including the proof of Theorem~\ref{thm:introfibration}. In Section~\ref{sec:JacobiansK3}, we study the general geometry of $K_A(0,l,s)$ and the fixed-point locus, and then focus on the elliptic fibers of the K3 surface in Section~\ref{sec:Jacobianssingular}. In Section~\ref{sec:JacobiansIsolatedpts}, we describe the isolated points in the fixed-point locus.

\subsubsection*{Acknowledgments}

We heartily thank Nicolas Addington, for pointing us in the direction of this work, and Alexander Polishchuk, for fruitful conversations during the early stages of this project.
We also thank Nils Bruin, Maria Fox, Lie Fu, Zhiyuan Li, Eyal Markman, Martin Olsson, John Voight, and Chelsea Walton for helpful discussions.

During the preparation of this article, S.F.~was partially supported by NSF DMS-1745670, and K.H.~was partially supported by NSERC.

\subsubsection*{Notation}
We write the standard involution on an abelian surface $A$, the morphism multiplying by $-1$ in the group law of $A$, as $\iota \colon A\to A$. 
We write $K_n(A)$ for the generalized Kummer variety of dimension $2n$. In particular, we write $K_1(A)$ for the Kummer K3 surface of $A$. 

For a smooth projective variety $X$ over a field $k$, let $X_{\kbar}:=X\times_k \kbar$. We denote by $\widetilde{H}(X_{\kbar},\Z_\ell)$ the $\ell$-adic Mukai lattice of $X$, which is the direct sum of the even cohomology twisted into weight zero: 
\[\widetilde{H}(X_{\kbar},\Z_\ell):=\textstyle\bigoplus_{i=0}^{\dim X} H^{2i}_{\et}(X_{\kbar}, \Z_\ell(i)).\]
This lattice is given the usual Mukai pairing, e.g. for $X=A$ an abelian surface, $(\alpha, \beta) =-\alpha_0\beta_4+\alpha_2\beta_2-\alpha_4\beta_0$.
We will always assume that our Mukai vectors $v$ satisfy the conditions given in Setting~\ref{defKv}, unless indicated otherwise.

Throughout, $D(X)$ denotes the bounded derived category of coherent sheaves on $X$.

%% file: backgroundmoduli.tex
Let $A$ be an abelian surface defined over an arbitrary field $k$. 

\begin{defn}
Let $\omega \in H^4_{\et}(A_{\kbar}, \Z_\ell(2))$ be the numerical equivalence class of a point on $A_{\kbar}$. A \emph{Mukai vector on $A$} is an element of
$$N(A):=\Z\oplus \mathrm{NS}(A)\oplus \Z\omega,$$
where $N(A)$ is a subgroup of $\widetilde{H}(A_{\kbar}, \Z_\ell)$ under the natural inclusion.

Given a coherent sheaf $\fF$ on $A$, we assign to it a Mukai vector $v(\fF)\in N(A)$ given by its rank, the N\'eron-Severi class of its determinantal line bundle, and its Euler characteristic. We will write this as $v(\fF)=(r,l,s)$.
\end{defn}

By fixing a Mukai vector $v$ and a polarization $H$ on $A$, we can construct the moduli space $M_{A,H}(v)$ parametrizing $H$-semistable sheaves on $A$. We use the more compact notation $M_A(v)$. 
We ask that the Mukai vector satisfy the following conditions in order to ensure that the moduli space is nicely behaved, i.e.~is a non-empty, smooth, projective variety over~$k$.

\begin{defn}
\begin{enumerate}[(a)]  
\item A Mukai vector $v \in N(A)$ is \emph{geometrically primitive} if its image under $N(A) \to N(A_{\kbar})$ is primitive in the lattice.

\item A Mukai vector $(r,l,s)$ is \emph{positive} if one of the following is satisfied:
\begin{enumerate}[(i)]
\item  $r>0$
\item  $r=0$, $l$ is effective and $s\neq 0$ 
\item  $r=0$, $l=0$ and $s<0$.
\end{enumerate}  

\item A polarization $H\in \mathrm{Pic}(A)$ is \emph{v-generic} if every $H$-semistable sheaf with Mukai vector $v$ defined over $\kbar$ is $H$-stable.
\end{enumerate}
\end{defn}

A polarization is often $v$-generic if it is not contained in a locally finite union of certain hyperplanes in $\NS(A_{\bar{k}})_{\R}$ defined in \cite[Def.\ 4.C.1]{HL}, but this is not always enough to ensure genericity (see for example, \cite[Ex.~1.7]{Frei}).

When $v^2=0$ and $H$ is $v$-generic, Mukai showed that $M_H(v)$ is an abelian surface \cite[Rmk.~5.13]{Mukaiexposition}. 
We focus on the higher-dimensional case.

\begin{prop}
  \label{modulinice}
Let $v\in N(A)$ be a geometrically primitive and positive Mukai vector with $v^2\geq 2$, and let $H$ be a $v$-generic polarization on $A$. Then $M_A(v)$ is a non-empty, smooth, projective, geometrically irreducible variety of dimension $v^2+2$ over $k$.
\end{prop}

\begin{proof}
The projectivity and smoothness are shown in \cite[Prop.~6.9]{FuLi},
which relies
on classic results in \cite{Mukaisymp} as well as \cite{La} for the construction of moduli spaces of semistable sheaves over arbitrary fields.
Geometric irreducibility of $M_A(v)$ follows from \cite[Thm.~4.1]{KLS} (note that the authors work over $\C$, but their proof holds for any algebraically closed field).
Finally, the dimension claim follows from \cite[Cor.~0.2]{Mukaisymp} once we know $M_A(v)$ is non-empty; non-emptiness is a consequence of \cite[Thm.~0.1]{Yoshioka} along with a lifting argument as in \cite[Prop.~6.9]{FuLi} when the field has positive characteristic.
\end{proof}

Let $v:=(r,l,s)$ be a Mukai vector as in Proposition~\ref{modulinice} and let
\[\Phi_P\colon D(A)\to D(\hat{A})
  \]
denote \textit{the Fourier--Mukai transform} on $A$, which has kernel the Poincar\'e bundle $P$ on $A \times \hat{A}$.
In \cite[Thm.~4.1]{Yoshioka}, Yoshioka proves over $\C$ that the Albanese variety of $M_H(v)$ is $A\times\hat{A}$ and fixing any $\fF_0\in M_H(v)$, we define the Albanese morphism as follows:
\begin{align}\label{albvar}
M_A(v) &\to \hat{A}\times A\\
\fF  &\mapsto (\det(\fF)\otimes \det(\fF_0)^{-1},
\det(\Phi_P(\fF))\otimes \det(\Phi_P(\fF_0))^{-1})\notag
\end{align}
This construction also shows that over an arbitrary field $k$,
the following map gives the
Albanese torsor  of $M_H(v)$: 
\begin{align}\label{albtor}
\alb\colon M_A(v) &\to \Pic^l_A\times\Pic^{m}_{\hat{A}}\\
\fF  &\mapsto (\det(\fF),\det(\Phi_P(\fF))),\notag
\end{align}
where $m$ is the N\'eron-Severi class in the Mukai vector $\Phi_P(v):=(s,m,r)$, which is the negative of the Poincar\'e dual of $l$ by \cite[Prop.~1.17]{MukaiFF}.

\begin{setting}\label{defKv}
Let $A$ be an abelian surface defined over a field $k$.  
Let $v:=(r,l,s)\in N(A)$ be a geometrically primitive and positive Mukai vector with $v^2\geq 6$ and $\Char k \nmid \frac{v^2}{2}$. Let $H$ be a $v$-generic polarization on $A$.
Fix $(L,M)$ a pair of line bundles in $\Pic^l(A)\times\Pic^{m}(\hat{A})$. 
Let $K_A(v)$ be the fiber of $\alb$ over $(L,M)$.
\end{setting}

Over $\C$, \cite[Thm.~0.2]{Yoshioka} shows that $K_A(v)$ is a hyperh\"ahler variety, and the following result generalizes this to other fields.

\begin{prop}[{{\cite[Thm.~0.2]{Yoshioka}, \cite[Prop.~6.9]{FuLi}}}]\label{prop:Kvsmooth}
Suppose we have data as in Setting~\ref{defKv}. 
Then $K_A(v)$ is a smooth, projective symplectic variety of dimension $v^2-2$ and is deformation equivalent to the generalized Kummer variety $K_{(v^2-2)/2}(A)$.
\end{prop}

For $K_A(v)$ over a field of characteristic zero, which we may assume is a subfield of $\C$, $K_A(v)_\C$ is a hyperk\"ahler variety. 
In positive characteristic, Fu and Li \cite[Def.~3.1]{FuLi} define a symplectic variety $X$ to be a smooth connected variety where $\pi_1^{\et}(X)=0$ and $X$ admits a nowhere degenerate closed algebraic $2$-form.

We are interested in symplectic involutions on $K_A(v)$. We will show in Theorem \ref{thm:symplectic} that these all involve
the induced action of the standard involution $\iota$ on $A$. 
Pullback $\iota^*$ sends degree $0$ line bundles on $A$ to their inverses. For any line bundle $\L\in\Pic(A)$,
the multiplication by $n$ map has the property that $[n]^*\L\cong \L^{n^2}\otimes M$ for some $M\in \Pic^0(A)$.
Thus $\L$ and~$\iota^*\L$ differ by a degree $0$ line bundle, so are always in the same N\'eron-Severi class.

In order for $\iota^*$ to give a well-defined morphism on $K_A(v)$,
$K_A(v)$ must be a fiber of the Albanese morphism over a pair of symmetric line bundles $L$ and $M$, which we prefer to do when possible for notational simplicity.
In the case of generalized Kummer varieties $K_{n-1}(A)$ or varieties $K_A(v)$ whose Mukai vector has trivial N\'eron-Severi class, it is always possible to choose the fiber over the structure sheaves of $A$ and $\hat{A}$.
For other choices of Mukai vector,
we show in Lemma~\ref{symmetric} below that
over an algebraically closed field we may always
choose such a pair of symmetric line bundles.

\begin{lemma}\label{symmetric}
Let $A$ be an abelian variety over an algebraically closed field~$k$. Then any class in $\NS(A)$ has a symmetric representative. Moreover, there is a short exact sequence of the following form, where $\Pic^{\sym}(A)$ is the subgroup of all symmetric line bundles:
  \[
0\to \Pic^0(A)[2]\to \Pic^{\sym}(A)\to \NS(A)\to 0.
  \]    
\end{lemma}  

\begin{proof}
The action of $\iota^*$ on the following short exact sequence
\[
0\to\Pic^0(A)\to \Pic(A)\to \NS(A)\to 0
\]    
gives rise to the long exact sequence
\[
0\to\Pic^0(A)[2]\to \Pic(A)^{\sym}\to \NS(A)\to H^1(\Z/2\Z, \Pic^0(A))\to \cdots,
\]
where $\NS^{\sym}(A)=\NS(A)$
since, for any line bundle $\L$, $\iota^*\L$ is in the same N\'eron-Severi class as $\L$. 
The group $H^1(\Z/2\Z, \Pic^0(A))$ is trivial since crossed homomorphisms correspond to elements in $\Pic^0(A)$ and principal crossed homomorphisms correspond to choices of element in $\Pic^0(A)$ that have a square root, which is all of them, since we are working over an algebraically closed field.
\end{proof}  

The proof above requires the field $k$ to be algebraically closed, but we will often work over a non-closed field. In that case, the existence of a symmetric line bundle in a given N\'eron-Severi class is not guaranteed. 
Rather than working over a finite extension of the ground field in order to acquire a symmetric bundle, we will simply alter $\iota^*$ by a correction factor to get an associated involution on $K_A(v)$ (see Construction~\ref{kappa}).

%% file: involutions.tex
In \cite[Cor.~5(2)]{HigherEn}, the authors show that, for $X=K_{n-1}(A)$ over $\C$, the kernel of 
\[\nu\colon \Aut X \to \mathrm{O}(H^2(X,\Z))\]
is isomorphic to $\Z/2\Z \ltimes (\Z/n\Z)^4$, generated by $\iota$ and translation by elements of $A[n]$.
In fact, this group of automorphisms is isomorphic to $\Z/2\Z \ltimes (\Z/n\Z)^4$ for any hyperk\"ahler variety $X$ of $\Kum_{n-1}$-type, since it is a deformation invariant \cite[Thm.~2.1]{HasTsc}.
Moreover, when $\dim X=4$, $\ker \nu$ contains all of the symplectic involutions
\cite[Thm.~7.5(i)]{KapferMenet}.
Markman identifies a subgroup $\Gamma_v\cong (\Z/n\Z)^4$ of
$\ker \nu$ when $X=K_A(v)$ as coming from the kernel of $\phi$ defined below
\cite[\S10.1]{MaKummers}.
In this section, we give an explicit description of $\ker \nu$ for $K_{A}(v)_{\bar{k}}$ when we are in the more general Setting~\ref{defKv} and $k$ is arbitrary; this will allow us to understand the action of the Galois group on the fixed-point loci of the involutions in $\ker \nu$.

In Section~\ref{sec:automorphisms}, we identify which automorphisms of $M_A(v)$ given by translation and tensoring by a degree 0 line bundle restrict to automorphisms of $K_A(v)$ and show they form a group isomorphic to $(\Z/n\Z)^4$. We also identify the group of such automorphisms
when $v$ is not primitive.
The other automorphism needed to generate $\ker \nu$
is $\iota^*$ when $K_A(v)$ is the Albanese fiber over symmetric line bundles; in Section~\ref{sec:alt_iota}, we
produce an involution $\kappa$ to replace $\iota^*$ in the more general setting. We then study the fixed loci of the compositions of $\kappa$ with the automorphisms produced in Section~\ref{sec:automorphisms}.
In Section~\ref{sec:symplectic}, we show 
that 
these compositions are symplectic and act trivially on $H^2(K_A(v),\Z)$.

\subsection{Automorphisms from translation and tensor}
\label{sec:automorphisms}

In this section we work with data as in Setting~\ref{defKv} with the additional assumption that $k$ is an algebraically closed field,
and we define $n:=\frac{v^2}{2}$.
Because $k=\bar{k}$ and $\Char k \nmid n$, we have $A[n]\cong (\Z/n\Z)^{4}$. 

We recall that given a line bundle $\mathcal{L}\in \Pic(A)$,
$\phi_{\L}\colon A\to \hat{A}$ is defined by $\phi_{\L}(x):=t_x^*\L\otimes \L^{-1}$, where $t_x\colon A\to A$ is translation by a point $x\in A$. 
We denote by 
$L_y\in \Pic^0(A)$ the line bundle corresponding to a point $y\in \hat{A}$.
Note that $\phi_{\L}$ is dependent only on the N\'eron-Severi class of $\L$, so we will use the notation $\phi_{[\L]}$.

Pullback by the translation map and tensoring by degree~$0$ line bundles give automorphisms of $M_A(v)$, and we are interested in when these automorphisms respect the Albanese morphism. That is, we identify in Theorem~\ref{n4} below which of the $L_y\otimes t_x^* \in \Aut M_A(v)$
restrict to automorphisms of $K_A(v)$. 

\begin{thm}\label{n4}
Let $v$ be a Mukai vector as in Setting~\ref{defKv}.  
There are exactly $n^4$ elements
$(x,y)\in A\times\hat{A}$
for which the automorphism $L_y\otimes t_x^*$ on $M_A(v)$
restricts to an automorphism on $K_A(v)$.
These elements form a subgroup $$G_A(v)\leqslant (A\times\hat{A})[n],$$ whose set of $k$-points is
isomorphic to $(\Z/n\Z)^4$.

The elements of $G_A(v)$ are the solutions
to the following equations on $\hat{A}$ and $A$, where $l$ and $m$ are the N\'eron-Severi classes of $L$ and $M$:
\begin{equation}\label{LM}
\phi_l(x)=-ry   \quad\text{and}\quad   \phi_m(y)=sx.
\end{equation}
Equivalently, $G_A(v)$ is the kernel of the following isogeny:
\begin{align}\label{phiisogeny}
  \phi\colon A\times\hat{A} &\to A\times\hat{A} 
  \\
(x,y)&\mapsto (\phi_m(y)-sx,\phi_l(x)+ry).\notag
\end{align}
\end{thm}

The proof of Theorem~\ref{n4} requires analysis of $\phi_l$ and $\phi_m$.
We will crucially need the following lemma:

\begin{lemma}[{{Yoshioka \cite[Lem.~4.2]{Yoshioka}}}]\label{yl}
$$\phi_m\circ\phi_l=-\chi \cdot 1_A \quad\text{and}\quad \phi_l\circ\phi_m=-\chi \cdot 1_{\hat{A}},$$
where $\chi:=\chi(L)=\chi(M)=\frac{l^2}{2}=n+rs$.
\end{lemma}  

Additionally, we recall that for any $\fF\in D(A)$,
 $$\Phi_P(t_x^*\fF)=L_{-x}\otimes\Phi_P(\fF)
 \quad\text{and}\quad
 \Phi_P(\fF\otimes L_y)=t_y^*\Phi_P(\fF).
 $$ 
This follows from \cite[(3.1)]{Mukai}. Though the statement is not quite identical to the one we give here, it immediately follows from biduality of the Poincar\'e bundle \cite[9.12]{HuyFMbook}.

\begin{proof}[Proof of Theorem~\ref{n4}]
The main issue in this proof is that maps of the form $L_y\otimes t_x^*$ are not in general well-defined as automorphisms on $K_A(v)$.   
Given $\fF\in K_A(v)$, $L_y\otimes t_x^*\fF$ has the same Mukai vector as $\fF$, but may not have the same image under the Albanese morphism. For instance pullback by $t_x^*$ in general preserves N\'eron--Severi classes of line bundles, and acts trivially on the structure sheaf, but it does not act trivially on all line bundles.

We therefore seek the $(x,y)\in A\times\hat{A}$ that satisfy the following conditions:
\begin{align*}
L&=\det(\fF)=\det(L_y\otimes t_x^*\fF)=L_y^{\otimes r}\otimes t_x^*\det(\fF)=
   L_y^{\otimes r}\otimes t_x^*L\\
  M&=\det(\Phi_P(\fF))=\det(\Phi_P(L_y\otimes t_x^*(\fF)))=
     \det(t_y^*(L_{-x}\otimes \Phi_P(\fF)))\\
  &\phantom{{}=\det(\Phi_P(\fF))}=
     t_y^*(L_{-x}^{\otimes s}\otimes\det(\Phi_P(\fF)))
       =t_y^*(L_{-x}^{\otimes s}\otimes M ) 
       =L_{-x}^{\otimes s}\otimes t_y^* M.
\end{align*} 
We may rewrite these conditions as the equations~\eqref{LM}. Equivalently, 
these $(x,y)$ are the kernel of the map $\phi$ in~\eqref{phiisogeny}.

Precomposing the map $\phi$ with
$\psi\colon A\times\hat{A} \to A\times\hat{A}$, where
$\psi(x,y)= (\phi_m(y)-rx,\phi_l(x)+sy)$, and applying Lemma~\ref{yl}, we have

\begin{align*}
\phi\circ\psi(x,y)&=
\phi\circ(\phi_m(y)-rx,\phi_l(x)+sy)\\
&=(\phi_m(\phi_l(x)+sy)-s(\phi_m(y)-rx),\phi_l(\phi_m(y)-rx)+r(\phi_l(x)+sy))\\
&=(-\chi\cdot x+rsx,-\chi\cdot y+rsy)=-n(x,y).
\end{align*}
Thus $\phi\circ\psi=[-n]$, so $\phi$ is surjective, 
hence an isogeny.
Similarly, $\psi\circ\phi=[-n]$ and 
$G_A(v)\leqslant  (A\times\hat{A})[n]$.

We show
$G_A(v)\cong (\Z/n\Z)^4$ in 
Lemma~\ref{beef}. 
This will require an understanding of preimages of elements under $\phi_l$ and $\phi_m$, which we study in Claims~\ref{kernel} and~\ref{helper}.
\end{proof}

\begin{rmk}\label{respectalllinebundles}
Since the maps $\phi_l$ and $\phi_m$
are determined by the N\'eron--Severi classes of $L$ and $M$,
the proof of Theorem~\ref{n4} shows that the automorphisms
of $M_A(v)$ given by elements of $G_A(v)$
will restrict to automorphisms of not just one, but any, fiber of the Albanese morphism on $M_A(v)$.

Furthermore, for any $(x,y)\in (A\times\hat{A})[n]$,
the automorphism $L_y\otimes t_x^*$ induces a permutation of the Albanese fibers and if $(x,y)\notin G_A(v)$, this permutation does not have any fixed fibers.

If we extend the domain of $\det \times \det \Phi_P$ 
to elements of $D(A)$ with Mukai vector $v$ (by mapping to the Grothendieck group before taking determinants), $L_y\otimes t_x^*$ acts on the fibers of this map as well.
\end{rmk}  

Before proving Lemma~\ref{beef},
we need results on the kernels of $\phi_l$ and $\phi_m$:

\begin{claim}\label{kernel}
Let $p\neq \charr k$ be a prime and $\chi\neq 0$.
Suppose $p^q$ is the highest power of $p$ dividing $\chi$.
Then the group of $p$-power torsion points in $\ker\phi_l\cong \ker\phi_m$ is\[
(\Z/p^{n_1}\Z)^2\times (\Z/p^{n_2}\Z)^2,
\]
where $0\leq n_1\leq n_2$ and $n_1+n_2=q$.
If $n_1>0$, then
$L$ and $M$ are $p^{n_1}$-st powers of other line bundles.
\end{claim}

If $L$ and $M$ are separable, we may define their 
\textit{polarization type} to be the termwise product of pairs $(p^{n_1},p^{n_2})$ as $p$ varies over primes dividing $\chi$ (cf.~\cite[\S2]{BLseparable}).

\begin{proof}
Since $\ker \phi_l \cong \ker \phi_{-l}$ and $\chi(L) \neq 0$, we may assume that $L$ is ample. The proof of Riemann--Roch for abelian varieties in \cite[\S16]{Mumford} implies that the degree of $\phi_l$ is $\chi^2$. 
The structure of $\ker \phi_l \cap A[p^q]$ is then determined by
Lemma~\ref{yl} and 
the fact that the Weil pairing $e^L$ on the $p$-torsion is skew-symmetric \cite[\S20, Thm.~1]{Mumford}. 
Since
$\phi_m$ is the negative of the dual of $\phi_l$ \cite[\S2]{BLseparable}, the group structure of $p$-power torsion points in $\ker \phi_m$ is isomorphic to that in $\ker \phi_l$.
The last statement is a consequence of \cite[\S23, Thm.~3]{Mumford}.
\end{proof}

The images of any two elements of the same order under 
the compositions $\phi_l\circ\phi_m$ or $\phi_m\circ\phi_l$ will have the same order. However, $\phi_l$ and $\phi_m$ do not respect orders in this way.

\begin{claim}\label{helper}
Let $p\neq \Char k$ be a prime dividing $\chi$, and assume that $l$ and $m$ are not $p$-th multiples of other classes, so $n_1=0$.
Suppose $p^d\mid\chi$ for some $d\in \mathbb{N}$. 
\begin{enumerate}[(a)]
\item Suppose $u\in A[p^d]\cap \ker{\phi_l}$. Then the preimage of $u$ in $\hat{A}[p^d]$ under $\phi_m$ is of the form
  $b+(\Z/p^d\Z)^2$ for some $b\in \hat{A}[p^d]$.
\item Suppose $v\in \hat{A}[p^d]\cap \ker{\phi_m}$. Then the preimage of $v$ in $\hat{A}[p^d]$ vnder $\phi_l$ is of the form
  $a+(\Z/p^d\Z)^2$ for some $a\in {A}[p^d]$.
\end{enumerate}
Now, suppose $p^q=\chi$. 
\begin{enumerate}[(c)]  
\item Suppose $u\in A$ and $\phi_l(u)$ has order $p^c$.
  Then the preimage of $u$ in $\hat{A}$ under $\phi_m$ is of the form
  $b+(\Z/p^q\Z)^2$ for some $b\in\hat{A}[p^{c+q}]$.
\item[(d)] Suppose $v\in \hat{A}$ and $\phi_m(v)$ has order $p^c$.
  Then the preimage of $v$ in ${A}$ under $\phi_l$ is of the form
  $a+(\Z/p^q\Z)^2$ for some $b\in\hat{A}[p^{c+q}]$.  
\end{enumerate}
\end{claim}

\begin{proof}
(a) By Lemma~\ref{yl},
the composition $\phi_l\circ\phi_m$ is given by multiplication by $-\chi$.
Thus, $\phi_m\circ\phi_l$ acts on $A[p^d]$ as the zero map, and hence:
$$\im \phi_m|_{\hat{A}[p^d]} \subseteq A[p^d]\cap \ker \phi_l.$$
By Claim~\ref{kernel}, $A[p^d]\cap \ker \phi_l$ has $p^{2d}$ elements and
$\phi_m$ acting on $\hat{A}[p^d]$ is a $p^{2d}$-to-$1$ map. It follows by counting that $\im \phi_m|_{\hat{A}[p^d]} = A[p^d]\cap \ker \phi_l$. By Claim~\ref{kernel}, the preimage of $u$ is as stated. 

Part (b) follows analogously.

(c) By Lemma~\ref{yl}, the preimage of $\phi_l(u)$ under $\phi_l\circ\phi_m$ consists of elements of order $p^{c+q}$. By Lemma~\ref{kernel}, the result follows.

Part (d) follows analogously.
\end{proof}

The following result is proved using a case-by-case argument. 
The explicit argument given has the advantage of aiding in the analysis of examples. See \cite[Lemma 10.1]{MaKummers} for an approach
using deformations over $\C$.

\begin{lemma}\label{beef}
  The solutions to the equations \eqref{LM} 
  form a group isomorphic to $(\Z/n\Z)^4
\leqslant (A\times\hat{A})[n]$.
\end{lemma}

\begin{proof}
\medskip\noindent\textit{Case 1: $\chi=0$.}

Both $L$ and $M$ must have degree $0$, 
so $\phi_l$ and $\phi_m$ are both the $0$-morphism. The equations \eqref{LM} simplify to:
\[0=-ry \quad\text{and}\quad 0=sx
\]
Furthermore, $n=-rs$. Since $v=(r,l,s)$ is positive and $v^2\geq 4$, we must have $r>0$ and $s<0$.
The solutions consist of all products of $|s|$-torsion points on $A$ and
$r$-torsion points on $\hat{A}$.

The group of solutions is isomorphic to $(\Z/r\Z)^4 \times (\Z/|s|\Z)^4$, hence
$(\Z/n\Z)^4$ since in this case, primitivity of the Mukai vector implies $r$ and $s$ are relatively prime.

\medskip

Now, let $p$ be a prime divisor of $n$ and $p^q$ be the highest power of $p$ dividing $n$.
We treat the remaining cases by analyzing solutions in $(A\times\hat{A})[p^q]$. We may then conclude by using the Sun Zi Remainder Theorem.

\medskip\noindent\textit{Case 2: $\chi\neq 0$ and at least one of $r$ or $s$ is relatively prime with $p$.}

Suppose $r$ is relatively prime with $p$.
Fix an arbitrary $x\in A[p^q]$.
The equation $\phi_l(x)=-ry$ then has exactly one solution $y$ because multiplication by $-r$ acts bijectively on $\hat{A}[p^q]$.

Now we check that $(x,y)$ is a solution to \eqref{LM}:
Applying $\phi_m$, we have $\phi_m\circ\phi_l(x)=-r\phi_m(y)$. Using Lemma~\ref{yl}, we then have $-rsx=-r\phi_m(y)$. Since $x$ and $y$ are $p^q$-torsion, multiplication by $-r$ acts bijectively, implying $sx=\phi_m(y)$.

Thus, for each $x\in A[p^q]$, there is one $y\in \hat{A}[p^q]$ so that $(x,y)$ is a solution to \eqref{LM}. The projection map $(x,y)\mapsto x$ gives an isomorphism from solutions to \eqref{LM} to $A[p^q]\cong (\Z/p^q\Z)^4$.

If $s$ is relatively prime with $p$, an analogous argument shows there is exactly one solution $(x,y)$ to \eqref{LM} for each $y\in \hat{A}[p^q]$ and that again the group of all solutions is isomorphic to  $(\Z/p^q\Z)^4$.

\medskip

Cases 1 and 2 have covered all cases where $r$ and $s$ are not both divisible by $p$. Going forward, we assume $p\mid r$ and $p\mid s$.
If $\charr(k)\neq 0$, our assumption in Setting~\ref{defKv} that $\charr(k)\nmid n$ implies in the following cases that $\charr(k)\neq p$ and so we may apply Claim~\ref{kernel}. By the primitivity of the Mukai vector, $n_1=0$ and 
$n_2$ is equal to the highest power of $p$ dividing $\chi$.

Let $j$ be the highest power of $p$ dividing $r$ and $k$ be the highest power of $p$ dividing $s$. If $r$ or $s$ is $0$, we choose the convention that $j$ or $k$ is $\infty$.

In each of Cases 3, 4, 5, we handle in stages the situations where $q$ becomes higher and higher relative to $j$ and $k$.
From now on, we assume $j\geq k$. If $k>j$ the argument is analogous.

\medskip\noindent\textit{Case 3: $\chi\neq 0$, $0<k\leq j$, and $q\leq j$.}
We observe that $p^q$ is the highest power of $p$ that divides~$\chi$.

Solutions $(x,y)\in (A\times \hat{A})[p^q]$ to the first equation in \eqref{LM} are precisely those where $\phi_l(x)=0$. By Claim~\ref{kernel}, the group of such $x$ is isomorphic to $(\Z/p^q\Z)^2$.

Fix such an $x$. We observe that $sx\in A[p^q]$ and $\phi_l(sx)=0$.
By Claim~\ref{helper}(a), the preimage of $sx$ under $\phi_m$ in $\hat{A}[p^q]$
is of the form $b+(\Z/p^q\Z)^2$ for some $b\in \hat{A}[p^q]$, thus there are $p^{4q}$ total solutions.

The projection $(x,y)\mapsto x$ gives a surjective group homomorphism 
$G_A(v)\twoheadrightarrow (\Z/p^q\Z)^2$.
The kernel of this map consists of all solutions where $x=0$, which by Claim~\ref{kernel} is isomorphic to $(\Z/p^q\Z)^2$.
Since $G_A(v)\leqslant (A\times\hat{A})[p^q]\cong (\Z/p^q\Z)^8$, this short exact sequence shows it must be isomorphic to $(\Z/p^q\Z)^4$.

\medskip

In Cases 4 and 5 we make a reduction argument.
We observe that for any $(x,y)\in G_A(v)$,
$(sx,sy)\in G_A(v)\cap (A\times\hat{A})[p^{q-k}]$.
In each of Cases 4 and 5, we will show that
the map
\begin{equation}\label{smult}
G_A(v)\xrightarrow{\cdot s} G_A(v)\cap (A\times\hat{A})[p^{q-k}]
\end{equation}
given by multiplication by $s$
is surjective and 
$p^{4k}$-to-$1$. This argument may be repeated to reduce each case to previous cases.

\medskip
\noindent\textit{Case 4: $\chi\neq 0$, $0<k \leq j< q$, and $q\leq j+k$.}

We note that $p^q\mid \chi$. 
Since $q-k\leq j$, the argument in Case 3 shows that 
\begin{equation}\label{case3}
G_A(v)\cap(A\times\hat{A})[p^{q-k}]\cong (\Z/p^{q-k}\Z)^4.  
\end{equation}

Let $(u,v)\in G_A(v)\cap (A\times\hat{A})[p^{q-k}]$.
We seek $(x,y)\in G_A(v)< (A\times\hat{A})[p^{q}]$ where $(sx,sy)=(u,v)$. First we search for elements $y$ where $\phi_m(y)=u$ and $sy=v$, thus we look at the preimage of $u$ under $\phi_m$ and analyze which of those elements give $v$ when multiplied by $s$.

Note that $\phi_l(u)=-rv=0$. Since $u$ is $p^{q-k}$-torsion, it is also $p^q$-torsion so by Claim~\ref{helper}(a), 
the preimage of $u$ under $\phi_m$
in $\hat{A}[p^q]$ is of the form $b+(\Z/p^q\Z)^2$ where $b\in\hat{A}[p^{q}]$.
Multiplying by $s$ gives a $p^{2k}$-to-$1$ map on the following cosets:
$$b+(\Z/p^q\Z)^2\xrightarrow{\cdot s} sb+(\Z/p^{q-k}\Z)^2.$$
We will now show that $v$ is in the image of this map:
the preimage of $su$ under $\phi_m$ in $\hat{A}[p^q]$
is of the form $v+(\Z/p^q\Z)^2$. The preimage of $su$ under $\phi_m$ that is $p^{q-k}$-torsion is thus of the form $v+(\Z/p^{q-k}\Z)^2$ and has exactly $p^{2(q-k)}$ elements.
Now, the elements of $sb+(\Z/p^{q-k}\Z)^2$ are $p^{q-k}$-torsion, there are $p^{2(q-k)}$ of them, and their image under $\phi_m$ is $su$, thus these sets are equal, implying $v\in sb+(\Z/p^{q-k}\Z)^2$.
Thus there are $p^{2k}$ elements $y\in \hat{A}[p^q]$ with the desired properties.

Now we search for elements $x$ where $\phi_l(x)=-ry=-\frac{r}{s}v$ and $sx=u$. Note that since $j\geq k$, $-\frac{r}{s}=\frac{cp^e}{d}$
for some $c$, $d$ relatively prime with $p$. We may define multiplying by $\frac{1}{d}$ on $p$-power torsion points by taking the preimage under multiplication by $d$ since it is a bijection on such points. 
We examine the preimage of $-\frac{r}{s}v$ under $\phi_l$ and analyze which of those elements give $u$ when multiplied by $s$.

Note $\phi_m(-\frac{r}{s}v)=-ru=0$, so by Claim~\ref{helper}(b), 
the preimage of $-\frac{r}{s}v$ under $\phi_l$ is of the form
$a+(\Z/p^q\Z)^2$ where $a\in {A}[p^{q}]$. 

Multiplying by $s$ gives a $p^{2k}$-to-$1$ map on the following cosets:
$$a+(\Z/p^q\Z)^2\xra{\cdot s} sa+(\Z/p^{q-k}\Z)^2.$$
We will now show that $u$ is in the image of this map:
The preimage of $-rv$ under $\phi_l$ in $A[p^q]$ is of the form $u+(\Z/p^q\Z)^2$. The preimage of $-rv$ under $\phi_l$ that is
$p^{q-k}$-torsion is thus of the form 
$u+(\Z/p^{q-k}\Z)^2$.
The elements of $sa+(\Z/p^{q-k}\Z)^2$ are $p^{q-k}$-torsion and their image under $\phi_l$ is $-rv$, thus these sets are equal, implying
$u\in sa+(\Z/p^{q-k}\Z)^2$.
In summary, there are $p^{2k}$ elements 
$x\in A[p^q]$  where $\phi_l(x)=-\frac{r}{s}v=-ry$ and $sx=u$.

This shows that
\eqref{smult} is a surjective $p^{4k}$-to-$1$ map.
Since multiplication by $s$ decreases the order of the $p$-power torsion of an element by exactly $p^k$, by 
\eqref{case3} we may conclude that
$G_A(v)\cong (\Z/p^q\Z)^4$.

\medskip
\noindent\textit{Case 5: $\chi\neq 0$, $0<k\leq j< q$, and $j+k<q$.}
In this case, $p^{j+k}$ divides $\chi$ and no higher powers of $p$ may divide $\chi$.

By the argument in Case 4, we have
\begin{equation}\label{case4}
G_A(v)\cap(A\times\hat{A})[p^{j+k}]\cong (\Z/p^{j+k}\Z)^4.
\end{equation}

We will first extend our result for solutions of order up to $p^{j+2k}$.
For convenience, define $t:=\min\{q,j+2k\}$.

Let $$(u,v)\in
G_A(v)\cap ((A\times\hat{A})[p^{t-k}]\setminus (A\times\hat{A})[p^{j}])
$$
We seek
$$(x,y)\in
G_A(v)\cap  (A\times\hat{A})[p^{t}]
$$
so that $(sx,sy)=(u,v)$. First we search for elements $y$ where $\phi_m(y)=u$ and $sy=v$; thus we look at the preimage of $u$ under $\phi_m$ and analyze which of those elements give $v$ when multiplied by $s$.

If $\phi_l(u)=0$, then the argument from Case~4 shows that there are $p^{2k}$ elements $y\in \hat{A}[p^{t}]$ where $\phi_m(y)=u$
and $sy=v$.

If $\phi_l(u)\in \hat{A}[p^k]\setminus \{0\}$, 
then by Claim~\ref{helper}(c),
the preimage of $u$ under $\phi_m$
in $\hat{A}[p^{t}]$
is of the form
$b+(\Z/p^{j+k}\Z)^2$ where $b\in \hat{A}[p^{t}]\setminus \hat{A}[p^{j+k}]$.

Note that $\phi_l(\phi_m(b))=-rv$ and by Lemma~\ref{yl}, $-rsb=-(n+rs)b=-rv$. 
Multiplication by $s$ gives a $p^{2k}$-to-$1$ map on the following cosets:
$$b+(\Z/p^{j+k}\Z)^2\xra{\cdot s}sb+(\Z/p^{j}\Z)^2.$$
We will now show that $v$ is in the image of this map.

The preimage of $su$ under $\phi_m$
in $\hat{A}[p^{j+k}]$
is of the form $v+(\Z/p^{j+k}\Z)^2$.
The part of $v+(\Z/p^{j+k}\Z)^2$
whose image under multiplication by $-r$ is $-rv$
is of the form $v+(\Z/p^{j}\Z)^2$. 
Since $\phi_m$ maps $sb+(\Z/p^{j}\Z)^2$ to $su$ and multiplying this coset by $-r$ gives $-rv$, 
by counting elements, $sb+(\Z/p^{j}\Z)^2=v+(\Z/p^{j}\Z)^2$, hence $v\in sb+(\Z/p^{j}\Z)^2$. Thus there are $p^{2k}$ elements $y\in b+(\Z/p^{j+k}\Z)^2$ where $\phi_m(y)=u$
and $sy=v$.

Now we search for elements $x$ where $\phi_l(x)=-ry=-\frac{r}{s}v$ and $sx=u$. We examine the preimage of $-\frac{r}{s}v$ under $\phi_l$ and analyze which of those elements give $u$ when multiplied by $s$.

If $\phi_m(-\frac{r}{s}v)=-ru=0$, we may conclude using the arguments in
Case~4.
Otherwise, by Claim~\ref{helper}(d),
the preimage of $-\frac{r}{s}v$ under $\phi_l$
in $A[p^{t}]$
is of the form
$a+(\Z/p^{j+k}\Z)^2$ where $a\in {A}[p^{t}]\setminus {A}[p^{j+k}]$.
Note that $\phi_m(\phi_l(a))=-ru$ and by Lemma~\ref{yl},
$-rsa=-ru$.

Multiplying by $s$ gives a $p^{2k}$-to-$1$ map on cosets:
$$a+(\Z/p^{j+k}\Z)^2 \xra{\cdot s} sa+(\Z/p^{j}\Z)^2.$$
We will now show that $u$ is in the image of this map.
The preimage of $-rv$ under $\phi_l$ is of the form $u+(\Z/p^{j+k}\Z)^2$.
The part of $u+(\Z/p^{j+k}\Z)^2$
whose image under multiplication by $-r$ 
is $-ru$ is of the form $u+(\Z/p^{j}\Z)^2$.
We have shown that $\phi_l$ maps $sa+(\Z/p^{j}\Z)^2$ to $-rv$ and multiplying this coset by $-r$ gives $-ru$. By counting elements, we have the set equality
$sa+(\Z/p^{j}\Z)^2=u+(\Z/p^{j}\Z)^2$, hence $u\in sa+(\Z/p^{j}\Z)^2$.
Thus there are
$p^{2k}$ elements $x\in a+(\Z/p^{j+k}\Z)^2$ where $\phi_l(x)=-\frac{r}{s}v$ and $sx=u$. Thus the following map is surjective and $p^{4k}$-to-$1$:
\begin{equation}\label{mults}
G_A(v)\cap (A\times\hat{A})[p^{t}]\xrightarrow{\cdot s} G_A(v)\cap (A\times\hat{A})[p^{t-k}]
\end{equation}
If $q\leq j+2k$, we may now conclude,
in combination with \eqref{case4},  that $G_A(v)\cong (\Z/p^q\Z)^4$.

If $q> j+2k$, \eqref{mults} shows that 
$G_A(v)\cap (A\times\hat{A})[p^{j+2k}]\cong (\Z/p^{j+2k}\Z)^4$.
The above argument may be repeated for solutions of orders up to 
$p^{j+3k}$ and then upward inductively to conclude that
$G_A(v)\cong(\Z/p^{q}\Z)^4$.
\end{proof}

\begin{example}\label{gpexample}
(a) For $K_2(A)\cong K_A(1,0,-3)$, $l$ and $m$ are the trivial N\'eron-Severi classes (these are treated in general by  Case 1 of the proof of Lemma~\ref{beef}),
so $\phi_l(x)=0$ and $\phi_m(y)=0$.
The equations \eqref{LM}
simplify to  $0=-y$ and $0=-3x$, which recovers the fact that the group of symplectic automorphisms for $K_2(A)$ is generated by $\iota$ and translation by elements of $A[3]$ \cite[Cor.~5(2)]{HigherEn}.

\smallskip

\noindent(b) In Sections~\ref{sec:JacobiansK3}-\ref{sec:JacobiansIsolatedpts}, we consider fourfolds $K_A(v)$ where $v=(0,l,s)$ for $l$ primitive and $\chi=3$.
If $s\equiv s' \textrm{\ mod\ } 3$, then $G_A(0,l,s)=G_A(0,l,s')$, leaving only three possible distinct groups of this form, which are described by a combination of
Cases~2 and 3 of Lemma~\ref{beef}. 
Case~2 shows that $G_A(0,l,1)$ and $G_A(0,l,2)$, though in general distinct, 
each have one element $(x,y)\in (A\times\hat{A})[3]$ for every $y\in \hat{A}[3]$, e.g.~for any $y\in \hat{A}[3]$, there is one $x\in\ker\phi_l$ so that $t_y^*M\simeq L_{x}\otimes M$.  However, we see from Case~3 that $G_A(0,l,0)$ is the product of $\ker\phi_l$ and $\ker\phi_m$.
\end{example}  

The assumption in Theorem~\ref{n4} that $v$ is primitive is necessary for $G_A(v)$ to be isomorphic to $(\Z/n\Z)^4$.
In the case where $v=2v_0$ for $v_0$ a primitive Mukai vector with $v_0^2=2$, which is used to construct O'Grady sixfolds, the solutions to the equations \eqref{LM} are precisely of the form $(A\times \hat{A})[2]\cong (\Z/2\Z)^8$, as is shown in
\cite[Lem.~5.1]{autoOG}. 
We generalize this result by extending
Theorem~\ref{n4}
to find all solutions to \eqref{LM} for any Mukai vector.

In the following result we alter our hypotheses by naming the primitive vector of Setting~\ref{defKv} $v_0$ and considering a multiple of $v_0$.

\begin{cor}
Let $v=(r,l,s)$ be a Mukai vector on an abelian surface $A$ so that $v=dv_0$
where $v_0=(r_0,l_0,s_0)$ is primitive and $n:=\frac{v_0^2}{2}$.

Then the group $G_A(v)$ of solutions $(x,y)\in A\times\hat{A}$ to the following equations 
\begin{equation}\label{unprimitive}
\phi_l(x)=-ry   \quad\text{and}\quad   \phi_m(y)=sx
\end{equation}
is isomorphic to $(\Z/dn\Z)^4\times(\Z/d\Z)^4$.
\end{cor}

\begin{proof}
Let $m$ and $m_0$ be the respective N\'eron-Severi classes determined by $\Phi_P$. 
We see from the definition of $\phi_l$ that $\phi_l=d\cdot \phi_{l_0}$ and likewise $\phi_m=d\cdot \phi_{m_0}$:
if we choose $(L_0,M_0)\in \Pic^{l_0}(A)\times \Pic^{m_0}(\hat{A})$ and  $L:=L_0^{\otimes d}$, $M:=M_0^{\otimes d}$, then for any $x\in A$, we have
\[
\phi_l(x):=t_x^*L\otimes L^{-1}=t_x^*L_0^{\otimes d}\otimes (L_0^{\otimes d})^{-1}
=(t_x^*L_0\otimes L_0^{-1})^{\otimes d}
=d\cdot \phi_{l_0}(x).
\]
Since $\phi_{l_0}$ and $\phi_{m_0}$ are group homomorphisms, we have, for any $(x,y)\in A \times \hat{A}$, 
\[
\phi_l(x)=\phi_{l_0}(dx)\quad\text{and}\quad
\phi_m(y)=\phi_{m_0}(dy).  
\]
Thus a pair $(x,y)\in A\times\hat{A}$ is a solution to \eqref{unprimitive}
if and only if 
\[
\phi_{l_0}(dx)=-r_0dy   \quad\text{and}\quad   \phi_{m_0}(dy)=s_0dx,
\]  
that is, $(dx,dy)$ solves the equations \eqref{LM} given by $v_0$.
By Theorem~\ref{n4} the set of solutions to the equations \eqref{LM} given by $v_0$ is isomorphic to $(\Z/n\Z)^4\cong G_A(v_0)$. We may conclude by observing that the set of solutions to
\eqref{unprimitive} is given by exactly the elements of $A\times\hat{A}$ that are in $G_A(v_0)$ after being multiplied by $d$. 
\end{proof}

\subsection{Involutions and fixed loci}
\label{sec:alt_iota}

Let $A$ be an abelian surface over an arbitrary field $k$.  
If $K_A(v)$ is a fiber over symmetric line bundles, then $\iota^*$ gives an involution of $K_A(v)$. However, if symmetric bundles do not exist in the appropriate N\'eron-Severi classes over $k$, we show here how to define an involution $\kappa$ to replace $\iota^*$. For the remainder of the section we fix a set of data as in Setting~\ref{defKv}, and hence fix a variety $K_A(v)$ over $k$.

We first give a lemma that will allow us to construct the involution $\kappa$.

\begin{lem}\label{relfibers}
Suppose we have an additional choice of line bundles 
$L'\in \Pic^l(A)$, $M'\in\Pic^m(\hat{A})$ over $k$.
Let $K_A(v)':=\alb^{-1}(L',M')$.
Then 
there is an element $(a,b)\in (A\times\hat{A})(k)$ so that $L_b\otimes t_a^*\colon K_A(v) \to K_A(v)'$ is an isomorphism over $k$.
It is unique up to composition with elements in $G_A(v)(k)$.
\end{lem}

\begin{proof}
Recall that for any $(x,y)\in A\times\hat{A}$, applying $L_y\otimes t_x^*$ to an element $\fF\in K_A(v)$, we have
  $$
\det(L_y\otimes t_x^*\fF)=L_y^{\otimes r}\otimes t_x^*L, \quad\text{and}\quad
\det(\Phi_P(L_y\otimes t_x^*\fF))=L_{-x}^{\otimes s}\otimes t_y^*M.
  $$
  We also recall that the morphism $\phi \colon A\times \hat{A} \to A \times \hat{A}$ from \eqref{phiisogeny} is an isogeny defined over $k$ and sends
  sends $(x,y)$ to
  $$
  (t_y^*M\otimes M^{-1}\otimes L_{-x}^{\otimes s},
t_x^*L  \otimes L^{-1}\otimes L_y^{\otimes r} ).
$$
The element $(a,b)$ desired is precisely a preimage of
$(L'\otimes L^{-1},M'\otimes M^{-1})\in (A\times \hat{A})(k)$
under $\phi$.
Finally, $L_b\otimes t_a^*: K_A(v) \to K_A(v)'$
is an isomorphism since 
it has an inverse $L_{-b}\otimes t_{-a}^*$.
\end{proof}

\begin{cons}\label{kappa}
Applying $\iota^*$ gives an isomorphism from $K_A(v)$ to 
$\alb^{-1}(\iota^*L,\iota^*M)$.
By Lemma~\ref{relfibers}, there is an $(a,b)\in (A\times\hat{A})(k)$ such that $L_b\otimes t_a^*$ maps isomorphically from $\alb^{-1}(\iota^*L,\iota^*M)$ back to $K_A(v)$, so we have the following automorphism defined over $k$:
\begin{align*}
\kappa\colon K_A(v)&\to K_A(v)\\
\fF&\mapsto  L_b\otimes t_a^*\iota^*\fF.
\end{align*}
\end{cons}

\begin{rmk}
We note that $\kappa$ is an involution. More generally, for any
$(c,d)\in A\times\hat{A}$, the morphism 
$L_d\otimes t_c^*\iota^*$ (which in general is an automorphism of $M(v)$ but perhaps not of $K(v)$) is an involution on $M_A(v)$. Indeed,
$(\iota\circ t_c)^2=\id$ on $A$ and $L_d$ is degree $0$, hence fixed under pullback by translation; thus for any $\fF\in M_A(v)$, we have:
\begin{align*}\label{invol}
(L_d\otimes t_c^*\iota^*)\circ (L_d\otimes t_c^*\iota^*)(\fF)
&=L_d\otimes t_c^*\iota^*L_d \otimes t_c^*\iota^*t_c^*\iota^*(\fF)\\
&=L_d\otimes t_c^*L_d^{-1} \otimes \fF
 =L_d\otimes L_d^{-1} \otimes \fF=\fF.
\end{align*}
\end{rmk}

The following are thus involutions of $K_A(v)$, where
$(x,y)\in G_A(v)(k)$:
\[
\kappa_{(x,y)}:= L_y\otimes t_x^*\kappa.  
\]
Under the simplifying assumption that $L$ and $M$ are symmetric, we may instead denote these involutions as 
\[
\iota_{(x,y)}:= L_y\otimes t_x^*\iota^*.
\]  

\begin{lemma}\label{translation}
Let $n:=\frac{v^2}{2}$ be odd.
Assume $k=\bar{k}$ and that $L$ and $M$ are symmetric, so $\iota^*$ is an involution on $K_A(v)$.
Then
$\Fix(\iota_{(x,y)})$ is a translation of $\Fix(\iota_{(0,0)})$, i.e., there exists $(u,w)\in G_A(v)$
so that:
\[
\Fix(\iota_{(x,y)})=L_{w}\otimes t_{u}^*(\Fix(\iota_{(0,0)})).\]
More generally, without the assumption that $L$ and $M$ are symmetric, there exists $(u,w)\in G_A(v)$ 
so that:
\[
  \Fix(\kappa_{(x,y)})=L_{w}\otimes t_{u}^*(\Fix(\kappa_{(0,0)})).\]
\end{lemma}

\begin{proof}
Let $\fF\in K_A(v)$ be in the fixed locus of $\iota^*$.
Pick $(u,w)\in G_A(v)$ so that $2w=y$ and $2u=x$, which is possible since $n$ is odd.
For instance when
$\frac{v^2}{2}=3$, $K_A(v)$ is a fourfold and the elements of $G_A(v)$ are all three-torsion, so we may choose $(-x,-y)$.

Then $L_{w}\otimes t_{u}^*\fF$ must be fixed by the involution
\[
L_{w}\otimes t_{u}^*\iota^*(L_{-w}\otimes t_{-u}^*)=L_{2w}\otimes t_{2u}^*\iota^*.
\]
The other direction of the containment is similar, as is the case with $\iota^*$ replaced by $\kappa$.
\end{proof}

\begin{prop}\label{kappa_fixed}
Let $n:=\frac{v^2}{2}$ be odd,
assume $k=\bar{k}$, and let 
$L'\in \Pic^l(A)$ and $M'\in\Pic^m(\hat{A})$ be symmetric line bundles (cf.~Lemma~\ref{symmetric}).
Fix an involution $\kappa$ as in Construction~\ref{kappa} on $K_A(v)$.
Then the fixed locus of $\kappa$ in $K_A(v)$
is isomorphic to the fixed locus of $\iota^*=\iota_{(0,0)}$ in $K_A(v)'$.
\end{prop}

\begin{proof}
By Lemma~\ref{relfibers}, there is an $(x,y)\in A\times \hat{A}$ so that
$L_y\otimes t_x^*$ gives an isomorphism from $K_A(v)$ to $K_A(v)'$.
The composition
$$
(L_{-y}\otimes t_{-x}^*)\circ \iota^*\circ (L_y\otimes t_x^*)
$$
may be rearranged to 
$L_{-2y}\otimes t_{-2x}^*\iota^*$, where
$L_{-2y}\otimes t_{-2x}^*$ gives an isomorphism from $\alb^{-1}(\iota^*L,\iota^*M)$  
to $K_A(v)$.
Thus, by the uniqueness statement of Lemma~\ref{relfibers}, there is an element $(u,w)\in G_A(v)$ for which $L_w\otimes t_u^*\circ L_{-2y}\otimes t_{-2x}^*$ is equal to 
the map $L_{b}\otimes t_{a}^*$ in the definition of $\kappa$. 
Then $L_{-y}\otimes t_{-x}^* \Fix(\iota^*)$ is equal to $\Fix(\kappa_{(-u,-w)})$, which is isomorphic to $\Fix(\kappa)$
by Lemma~\ref{translation}.
\end{proof}

Finally, we give the following general result on the action of the Galois group on the geometric fixed loci. 

\begin{prop}\label{prop:Galois}
Let $k$ be an arbitrary field. For $(x,y)\in G_{A_{\kbar}}(v)$, the action of $\sigma\in \Gal(k^{\sep}/k)$
sends the fixed
locus of $\kappa_{(x,y)}$ in $K_A(v)_\kbar$ to the fixed locus of $\kappa_{(\sigma^{-1} x,\sigma^{-1} y)}$.
\end{prop}  

\begin{proof}
Suppose $\fF$ is fixed by $\kappa_{(x,y)}$.
We use the equality
$t_x\circ\sigma=\sigma\circ t_{\sigma^{-1}x}$
and the observation that $\sigma$ commutes with $\iota$ and, moreover, $\kappa$, since
$\kappa$ is defined over the ground field $k$,
to simplify the following equation:
\[
\sigma^*\fF \simeq \sigma^*(L_y\otimes t_x^*\kappa\fF)
\simeq \sigma^*L_y\otimes \sigma^*t_x^*\kappa\fF 
\simeq \sigma^*L_y\otimes t_{\sigma^{-1} x}^*\kappa(\sigma^*\fF).
\]
Then we have $\sigma^*L_y\simeq L_{\sigma^{-1}y}$, which we may verify using $\Phi_P\colon D(A)\to D(\hat{A})$: 
\begin{align*}
\Phi_P(L_{-\sigma^{-1}y}\otimes \sigma^*L_y)
&\simeq t_{-\sigma^{-1}y}^*\sigma^*k(-y)[-g]
\simeq \sigma^*t_{-y}^*k(-y)[-g]\\
&\simeq\sigma^*k(0_{\hat{A}})[-g]\simeq k(0_{\hat{A}})[-g].  \qedhere
\end{align*}
\end{proof}

\subsection{Symplectic automorphisms and involutions}
\label{sec:symplectic}
Let $A$ be an abelian surface over $\C$. 
In the following lemma we give a generalization of \cite[Cor.~5(2)]{HigherEn} to hyperk\"ahler varieties $K_A(v)$ over $\C$: 

\begin{thm}\label{thm:symplectic}
Suppose we are in Setting~\ref{defKv} and we fix an involution $\kappa$ as in Construction~\ref{kappa}. 
Then the kernel of 
\begin{equation}\label{kerH}
  \nu \colon \Aut(K_A(v))\to \OO(H^2(K_A(v),\Z))
\end{equation}  
consists of automorphisms
of the form $L_y\otimes t_x^*$ and of the form $\kappa_{(x,y)}:=L_y\otimes t_x^*\kappa$ for $(x,y)\in G_A(v)$. 
Thus, for any $(x,y)\in G_A(v)$, the automorphism $L_y\otimes t_x^*$ is symplectic.  
The $\kappa_{(x,y)}$ are symplectic involutions of $K_A(v)$, and when $\dim K_A(v)=4$, these are all of the symplectic involutions.
\end{thm}

\begin{rmk}
While $\kappa$ is not unique, by
Lemma~\ref{relfibers}, the collection of elements in $\ker \nu$ is independent of the choice made in Construction~\ref{kappa}.
\end{rmk}

\begin{proof}
Elements of $(x,y)\in A\times\hat{A}$ act on $M_A(v)$ via $L_{y}\otimes t_{x}^*$. Abelian varieties are path-connected, so the action of any element in $A\times\hat{A}$ is homotopic to the identity, which implies the induced action on $H^2(M_A(v),\Z)$ is trivial.
If $(x,y)\in G_A(v)$,  then Theorem~\ref{n4} shows that the action of $L_{y}\otimes t_{x}^*$ restricts to  $K_A(v)$. By \cite[Thm.~0.2(2)]{Yoshioka}, the restriction map $H^2(M_A(v),\Z)\to H^2(K_A(v),\Z)$ is a surjection. Therefore, $L_{y}\otimes t_{x}^*$ acts trivially on $H^2(K_A(v),\Z)$ as well.

By \cite[Thm.~0.2(2)]{Yoshioka}, there is an isomorphism 
$$H^2(K_A(v),\Z)\cong v^\perp,$$
where $v^\perp \subset H^{even}(A,\Z)$ is the orthogonal complement to $v$ under the Mukai pairing.
Since $\iota^*$ acts by $-1$ on $H^1(A,\Z)$, it acts trivially on $H^{even}(A,\Z)$.

If we assume $L$ and $M$ are symmetric, $\iota^*$ is an automorphism of $K_A(v)$
and therefore must act trivially on $H^2(K_A(v),\Z)$. If $L$ and $M$ are not both symmetric, since we are working over an algebraically closed field, 
we observe that $\kappa$
is a composition of translation to an Albanese fiber over symmetric bundles, application of $\iota^*$ on that fiber, 
and translation back (cf.~proof of Proposition~\ref{kappa_fixed}), and thus $\kappa$ will act trivially on $H^2(K_A(v),\Z)$ as well.

By the discussion above,
$\ker \nu$ 
contains $2n^4$ elements, so by Theorem~\ref{n4} we have identified all of them.
The automorphisms in this kernel are clearly symplectic as the symplectic form generates part of $H^2(K_A(v),\C)$.

For any nontrivial choice of $(x,y)\in G_A(v)$, $L_y\otimes t_x^*$ is not an involution, but by Section~\ref{sec:alt_iota},
$\kappa_{(x,y)}$ is an involution on $K_A(v)$.

Finally, suppose $\dim K_A(v)=4$. By \cite[Thm.~7.5(i)]{KapferMenet}, all of the symplectic involutions of $K_A(v)$ act trivially on $H^2(K_A(v),\Z)$.
\end{proof}

%% file: cohomology.tex
In this section, we work with data as in Setting~\ref{defKv} with the additional assumption 
that $v^2=6$, so $K_A(v)$ is a fourfold. We will prove results characterizing
the middle cohomology of $K_A(v)$ when $k$ has  
characteristic~$0$ in Section~\ref{sec:charzero}. We use these results to characterize the cohomology similarly when $k$ has
positive characteristic in Section~\ref{sec:charp} via a brief lifting argument.

\subsection{Results in characteristic zero}\label{sec:charzero}

Assume
$K_A(v)$ is defined over an arbitrary field $k$ of characteristic zero, so
we may assume without loss of generality that $\bar{k} \hookrightarrow \C$. In this case we can identify the Galois representations which make up the middle cohomology of $K_A(v)$. 
This will depend on understanding the fixed loci of $\kappa_{(x,y)}$ for $(x,y)\in G_{A_\kbar}(v)$.

\begin{prop}\label{K3pts}
Suppose $k=\kbar$. 
The fixed locus of any
involution $\kappa_{(x,y)}$ for $(x,y)\in G_A(v)$ on a fourfold $K_A(v)$ consists of a K3 surface and 36 isolated points.
\end{prop}

\begin{proof}
First, suppose $k=\C$.
Work of Hassett and Tschinkel \cite{HasTsc} and Tar\'i \cite{Tari} shows that the statement is true for $K_2(A)$. A discussion of the isolated fixed points in this case is given in Section~\ref{isolated}.
  
Every hyperk\"ahler fourfold $K_A(v)$ is deformation equivalent to $K_2(A)$ and by \cite[Thm.~2.1]{HasTsc}, its group of symplectic involutions is also a deformation invariant. Thus, as in  Kapfer and Menet \cite[Thm.~7.5]{KapferMenet}, the fixed loci are related by deformation as well, so the statement holds for $K_A(v)$.

Now let $k$ be any algebraically closed field of characteristic zero. Since $A$ is defined over $k$, we can assume without loss of generality that $k \hookrightarrow \C$. Let ${K}_A(v)_{\C} := K_A(v) \times_{k} \C$ and consider the Cartesian square
\[
\xymatrix{{K}_A(v)_{\C} \ar[r]^{\widetilde{\kappa}_{(x,y)}} \ar[d] & {K}_A(v)_{\C} \ar[d] \\
K_A(v) \ar[r]^{\kappa_{(x,y)}} & K_A(v),
}\]
where $\widetilde{\kappa}_{(x,y)}$
is formed by replacing $\kappa$ with its extension to $\C$, which we call $\tilde{\kappa}$. 
By Theorem \ref{thm:symplectic}, $\widetilde{\kappa}_{(x,y)}$ is a symplectic involution, and by the argument above, $\Fix(\widetilde{\kappa}_{(x,y)})$ is a K3 surface $Z:=Z_{(x,y)}\subset {K}_A(v)_{\C}$ plus 36 isolated points. 

By \cite[Rmk.~3 following Thm~2.3]{Fogarty}, 
\[\Fix(\widetilde{\kappa}_{(x,y)})=\Fix(\kappa_{(x,y)})\times_{k} \C.\]
This descent of the fixed-point locus means that $\Fix(\kappa_{(x,y)})$ consists of a surface $S:=S_{(x,y)}\subset K_A(v)$ along with 36 $k$-points. We claim that $S$ is a K3 surface: indeed, we see via the valuative criterion of properness, using the fact that $\Fix(\kappa_{(x,y)})$ is a closed subscheme of $K_A(v)$, that $S \to \Spec k$ is proper.
By flat base change, we have that
$H^1(S,\mathcal{O}_S)\otimes \C \cong H^1(Z,\mathcal{O}_{Z})=0$, and $H^0(S,\omega_S)\otimes \C \cong H^0(Z,\omega_{Z})=\C$, so $\omega_S$ has a non-vanishing global section and hence is trivial. Finally, $S$ is smooth by \cite[Lem.~4.1]{Donovan}, which completes the proof.
\end{proof}

See \cite{KMO} for further discussion of these fixed-point loci in
hyperk\"ahlers of Kummer type.

\medskip

Let $k$ now be arbitrary. Let $S_{(x,y)}\subset K_A(v)_{\bar{k}}$ be the K3 surface in $\Fix(\kappa_{(x,y)})$ and $s_{(x,y)}\in H^4_\et(K_A(v)_{\bar{k}},\Q_\ell(2))$ the image of $[S_{(x,y)}]\in \CH^2 K_A(v)_{\bar{k}}$ under the cycle class map $\CH^2 K_A(v)_{\bar{k}} \to H^4_\et(K_A(v)_{\bar{k}},\Q_\ell(2))$. 

\begin{lemma}\label{lem:Galoiscohom}
For $\sigma \in \Gal(\bar{k}/k)$, the induced action on the cycle classes $s_{(x,y)}$ for $(x,y)\in G_{A_{\bar{k}}}(v)$ is given by
\[\sigma^*s_{(x,y)}=s_{(\sigma x, \sigma y)}\in H^4_\et(K_A(v)_{\bar{k}},\Q_\ell(2)).\]
\end{lemma}

\begin{proof}
By \cite[Prop.~9.2]{Milnebook}, the cycle class map is Galois equivariant, so $\sigma^*s_{(x,y)}$ is the cycle class of $[\sigma^* S_{(x,y)}]\in \CH^2 K_A(v)_{\bar{k}}$. As in the proof of \cite[Prop.~9.2]{Milnebook}, we have that $[\sigma^* S_{(x,y)}]$ is the preimage of $S_{(x,y)}$ under $\sigma^*\colon K_A(v)_{\bar{k}} \to K_A(v)_{\bar{k}}$. By Proposition~\ref{prop:Galois}, $(\sigma^*)^{-1}(S_{(x,y)})=S_{(\sigma x, \sigma y)}$. Thus we conclude that $\sigma^*s_{(x,y)}=s_{(\sigma x, \sigma y)}$, as desired.
\end{proof}

\begin{defn}
For a finite Galois module $G$, let $\Q_\ell[G]$ be the $\Q_\ell$-vector space with basis given by $G$, where the action of the Galois group on $\Q_\ell[G]$ is determined by the action on $G$: for $\sigma \in \Gal(\kbar/k)$ and $\sum_{g_i \in G} a_i g_i \in \Q_\ell[G]$,
\[\sigma \cdot \sum_{g_i \in G} a_i g_i = \sum_{g_i \in G} a_i \left(\sigma \cdot g_i\right).\]
We call $\Q_\ell[G]$ the \textit{permutation representation}.
\end{defn}

Recall that when $k$ is not algebraically closed, the group $G_{A_{\kbar}}(v)$ naturally has the structure of a finite $\Gal(\kbar/k)$-module.

\begin{thm}\label{thm:middlecohom}
There is an isomorphism of Galois representations
$$H^4_{\et}(K_A(v)_{\bar{k}},\Q_\ell(2))\cong \Sym^2H^2_{\et}(K_A(v)_{\bar{k}},\Q_\ell(1)) \oplus V,$$
where $V$ is the $80$-dimensional subrepresentation of $\Q_\ell[G_{A_{\bar{k}}}(v)]$ such that
\[\Q_\ell[G_{A_{\bar{k}}}(v)]\cong V \oplus \Q_\ell,\]
and the trivial representation $\Q_\ell$ is the span of $(0,0)\in G_{A_{\bar{k}}}(v)$.
\end{thm}

\begin{rmk}
As will be shown in Lemma~\ref{lem:H2}, the action of the Galois group on $H^2_{\et}(K_A(v)_{\bar{k}},\Q_\ell(1))$, and hence 
$\Sym^2H^2_{\et}(K_A(v)_{\bar{k}},\Q_\ell(1))$,
 is determined by the action on $H^2_{\et}(A_{\bar{k}},\Q_\ell(1))$.
\end{rmk}

\begin{proof}
  By Theorem~\ref{n4}, we have $3^4=81$ involutions
\begin{align*}
  \kappa_{(x,y)} \colon K_A(v)_{\bar{k}} &\to K_A(v)_{\bar{k}}\\
  \shF &\mapsto L_{y}\otimes t_{x}^*\kappa\shF
\end{align*}  
where $(x,y)\in G_{A_{\bar{k}}}(v)$. 

As in the proof of Proposition~\ref{K3pts}, let ${K}_A(v)_{\C} := K_A(v)\times_k \C$ and $\widetilde{\kappa}_{(x,y)}\colon {K}_A(v)_{\C} \to {K}_A(v)_{\C}$ the base change of $\kappa_{(x,y)}$. 
By Proposition~\ref{K3pts}, $\Fix(\widetilde{\kappa}_{(x,y)})$
contains
a K3 surface $Z_{(x,y)}\subset {K}_A(v)_{\C}$.
This gives 81 distinct K3 surfaces in ${K}_A(v)_{\C}$, where the distinctness follows from \cite[Thm.~2.1]{HasTsc}.  Via the cycle class map, these 81 surfaces give corresponding classes $z_{(x,y)}\in H^4({K}_A(v)_{\C}, \Q)$.

Similarly, 
there are 
K3 surfaces $S_{(x,y)}\subset K_A(v)_{\bar{k}}$ and corresponding cohomology classes $s_{(x,y)}\in H^4_{\et}({K}_A(v)_{\bar{k}}, \Q_\ell(2))$ such that $S_{(x,y)}\times_{\bar{k}} \C= Z_{(x,y)}\subset {K}_A(v)_{\C}$. 
Under the comparison and smooth base change isomorphisms
\[ H^4(K_A(v)_{\C}, \Q)\otimes_\Q \Q_\ell(2) \cong H^4_{\et}(K_A(v)_{\bar{k}}, \Q_\ell(2)),\]
the classes $z_{(x,y)}$ correspond to the classes $s_{(x,y)}$. 

By \cite[Thm.~7.5(ii)]{KapferMenet}, the pair $({K}_A(v)_{\C}, \widetilde{\kappa}_{(x,y)})$ is deformation equivalent to the pair $(K_2(A_{\C}), t_{\tau}\circ [-\Id]^{[[3]]})$ for some $\tau \in A_{\C}[3]$. In particular, these complex manifolds are diffeomorphic and so they have isomorphic cohomology rings. 
By \cite[Prop.~4.3]{HasTsc} (see also the discussion in \cite[\S6.4]{KapferMenet}), the $\Q_\ell$-span of $\{z_{(x,y)}-z_{(0,0)}\}_{(x,y) \in G_{A_\C}(v)}$ is an 80-dimensional vector space of $H^4(K_A(v)_{\C},\Q_\ell(2))$ which is a direct sum complement to the subspace $\Sym^2H^2(K_A(v)_{\C},\Q_\ell(1))$. 

Since the $s_{(x,y)}$ in $H^4(K_A(v)_{\bar{k}},\Q_\ell(2))$ correspond to
the $z_{(x,y)}$, it follows that \[V:=\Span_{\Q_\ell}\{s_{(x,y)}-s_{(0,0)}\}_{(x,y)\in G_{A_{\bar{k}}(v)}}\] 
is an $80$-dimensional subspace of $H^4_{\et}(K_A(v)_{\bar{k}},\Q_\ell(2))$ which is a direct sum complement to $\Sym^2H^2_{\et}(K_A(v)_{\bar{k}},\Q_\ell(1))$. 

By Lemma \ref{lem:Galoiscohom} we know that for $\sigma\in \Gal(\bar{k}/k)$,
$$\sigma^*(s_{(x,y)})=s_{(\sigma x,\sigma y)}.$$
Thus, 
$V$ is a Galois-invariant subspace of $H^4_{\et}(K_A(v)_{\bar{k}},\Q_\ell(2))$. Noting that $\Q_\ell[G_{A_{\bar{k}}}(v)]$ is semisimple by Maschke's Theorem---the Galois action factors through a finite group representation determined by the finite extension of $k$ over which $G_{A_\kbar}(v)$ is defined---and that $\sigma^*(s_{(0,0)})=s_{(0,0)}$, this shows that $V$ is the $80$-dimensional subrepresentation of $\Q_\ell[G_{A_{\bar{k}}}(v)]$ such that
\[\Q_\ell[G_{A_{\bar{k}}}(v)]\cong V \oplus \Q_\ell,\]
where the trivial representation corresponds to $(0,0)\in G_{A_{\bar{k}}}(v)$. Hence, $H^4_{\et}(K_A(v)_{\bar{k}},\Q_\ell(2))$ has the decomposition as stated. 
\end{proof}

\subsection{Results in positive characteristic via lifting}\label{sec:charp}
In
this section we observe that, 
because Kummer varieties $K_A(v)$
  defined over a field of positive characteristic lift to characteristic $0$ \cite{FuLi},
we may use Theorem~\ref{thm:middlecohom} to
give a similar description of
the middle cohomology. 

\begin{prop}\label{charp}
Suppose we have data as in Setting~\ref{defKv} where the base field
$k$ has characteristic $p>0$.
Then
$$H^4_{\et}(K_A(v)_{\bar{k}}, \Q_\ell(2))
\cong
\Sym^2H^2_{\et}(K_A(v)_{\bar{k}},\Q_\ell(1)) \oplus V',
$$
where $V'$ is the $80$-dimensional subrepresentation of $\Q_\ell[G_{A_{\bar{k}}}(v)]$ such that
\[\Q_\ell[G_{A_{\bar{k}}}(v)]\cong V' \oplus \Q_\ell,\]
and the trivial representation $\Q_\ell$ is the span of $(0,0)\in G_{A_{\bar{k}}}(v)$.
\end{prop}

\begin{proof}
  As explained in the proof of \cite[Prop.~6.9]{FuLi}, it is possible to lift $K_A(v)$ to characteristic $0$ by lifting its defining data.
That is,
the data $(A,v,H,L,M)$ defined over $k$  
has a lift
$(\mathcal{A}, v_W, \mathcal{H}, \mathcal{L}, \mathcal{M})$ to a complete discrete valuation ring $W$ of characteristic zero with residue field $k$ and field of fractions $F:=\Frac W$.
Note that all of this lifting data can be recovered from a lift of $(A,H,L)$: lifting $A$ automatically gives us a lift of $\hat{A}$, and lifting line bundles on $\hat{A}$ amounts to lifting their N\'eron--Severi class; a lift of the N\'eron--Severi class of $M$ is given by the N\'eron--Severi class of $\det(\Phi_{\mathcal P}(\mathcal{L}))$.
Call the specialization of $v_W$ to the generic fiber $v_F$.

There is a surjection of Galois groups
\begin{equation}\label{surj}
  \Gal(\bar{F}/F) \twoheadrightarrow \Gal(\bar{k}/k)
\end{equation}  
which is given by restricting automorphisms to the ring of integers of $\bar{F}$ and then passing to the quotient $\bar{k}$.
By the smooth base change theorem 
\cite[Exp.~XVI, Corollaire~2.2]{SGA}, for $\ell \neq p$
there are isomorphisms
\begin{equation}\label{smthbc}
\begin{aligned}
  H^2_{\et}(K_{A_{\bar{F}}}(v_F), \Q_\ell(1)) &\cong H^2_{\et}(K_{A}(v)_{\bar{k}}, \Q_\ell(1)),\text{ and}\\
  H^4_{\et}(K_{A_{\bar{F}}}(v_F), \Q_\ell(2)) &\cong H^4_{\et}(K_{A}(v)_{\bar{k}}, \Q_\ell(2)),  
\end{aligned}
\end{equation}
which are equivariant with respect to the action of $\Gal(\bar{F}/F)$ on the left and $\Gal(\bar{k}/k)$ on the right, compatible
with \eqref{surj}.

%By Tate's theorem \cite{TateThm}
%$H^2_{\et}(A_\kbar, \Ql)$ is a semisimple representation, and thus 
%so is $\Sym^2H^2_{\et}(K_A(v)_{\bar{k}},\Q_\ell(1))$.
%As argued in the proof of Theorem~\ref{thm:middlecohom}, $\Q_\ell[G_{A_{\bar{k}}}(v)]$ is semisimple as well, and thus $H^4_{\et}(K_{A}(v)_{\bar{k}}, \Q_\ell(2))$ is semisimple.
%
%Applying \eqref{smthbc} to this decomposition
%and using the semisimplicity of
%$H^4_{\et}(K_{A}(v)_{\bar{k}}, \Q_\ell(2))$,
%we see that 
%$\Sym^2H^2_{\et}(K_A(v)_{\bar{k}},\Q_\ell(1))$ is a summand of
%$H^4_{\et}(K_A(v)_{\bar{k}},\Q_\ell(2))$ and we may write
%$$H^4_{\et}(K_A(v)_{\bar{k}},\Q_\ell(2))\cong \Sym^2H^2_{\et}(K_A(v)_{\bar{k}},\Q_\ell(1)) \oplus V',$$
%where there is an isomorphism $V\cong V'$ 
%which is equivariant with respect to the action of $\Gal(\bar{F}/F)$ on the left and $\Gal(\bar{k}/k)$ on the right, again compatible with \eqref{surj}.
The isomorphisms of \eqref{smthbc} are compatible with the ring structure on cohomology, so the isomorphism $H^4_{\et}(K_{A_{\bar{F}}}(v_F), \Q_\ell(2)) \cong H^4_{\et}(K_{A}(v)_{\bar{k}}, \Q_\ell(2))$ restricts to an isomorphism 
\[\Sym^2 H^2_{\et}(K_{A_{\bar{F}}}(v), \Q_\ell(1)) \cong \Sym^2 H^2_{\et}(K_{A}(v)_{\bar{k}}, \Q_\ell(1)),\]
again compatible with the respective Galois group actions. 

Let the following be the decomposition given by Theorem~\ref{thm:middlecohom}:
$$H^4_{\et}(K_{A_{\bar{F}}}(v_F),\Q_\ell(2))\cong \Sym^2H^2_{\et}(K_{A_{\bar{F}}}(v_F),\Q_\ell(1)) \oplus V,$$
and let $V' \subset H^4_{\et}(K_{A}(v)_{\bar{k}}, \Q_\ell(2))$ be the vector space complement to \linebreak$\Sym^2 H^2_{\et}(K_{A}(v)_{\bar{k}}, \Q_\ell(1))$. Using the fact that $V$ is a $\Gal(\bar{F}/F)$ subrepresentation of $H^4_{\et}(K_{A_{\bar{F}}}(v_F), \Q_\ell(2))$, we conclude that
\[H^4_{\et}(K_A(v)_{\bar{k}}, \Q_\ell(2)) \cong \Sym^2 H^2_{\et}(K_A(v)_{\bar{k}}, \Q_\ell(1)) \oplus V'\]
as $\Gal(\bar{k}/k)$ representations. In particular, there is an isomorphism $V\cong V'$ which is equivariant with respect to the action of $\Gal(\bar{F}/F)$ on the left and $\Gal(\bar{k}/k)$ on the right, again compatible with \eqref{surj}.

The subgroup $G_{A_{\bar{F}}}(v_F)\leqslant (A_{\bar{F}}\times \hat{A}_{\bar{F}})[3]$ is given by
equations \eqref{LM} determined by $v_F$, which is part of our lifted data.
Thus, since the action of $\Gal(\bar{F}/F)$ on $V$ is given by $G_{A_{\bar{F}}}(v)$, the action of $\Gal(\bar{k}/k)$ on $V'$ must be the one determined analogously by
$G_{A_{\bar{k}}}(v)$. 
\end{proof}

%% file: derivedcategories.tex
There are a number of results related to derived equivalences of 
smooth, projective symplectic
varieties. For example,
if $X$ and $Y$ are derived equivalent smooth complex projective surfaces, then $D(\Hilb^nX)\cong D(\Hilb^nY)$ \cite[Prop.~8]{Ploog}.
If $X$ and $Y$ are K3 surfaces, then the converse holds,
and
  if two moduli spaces of stable sheaves $M_X(v)$ and $M_Y(v')$ are derived equivalent, then $X$ and $Y$ are also derived equivalent \cite[Cor.~9.7]{Beckmann}.
If $X$ and $Y$ are derived equivalent K3 surfaces over any field $k$, then
the $\ell$-adic \'etale cohomologies of
any moduli $M_X(v)$ and $M_Y(v')$ of equal dimension are isomorphic as $\Gal(\bar{k}/k)$ representations \cite[Thm.~2]{Frei}.
However, it is still an open question when
such moduli 
are derived equivalent. 

In the direction of symplectic varieties of Kummer type,
complex abelian surfaces $A$ and $B$ are derived equivalent if and only if there is an isomorphism $K_1(A)\cong K_1(B)$ between their associated Kummer K3 surfaces
\cite{HLOY,Stellari}. This result has also been proved for abelian surfaces
over fields of odd characteristic \cite{LiZou}; the relation between Kummer surfaces and twisted derived equivalence of abelian surfaces has been examined in \cite[Thm.~6.5.2]{twistLiZou}.
While $A$ and $\hat{A}$ are always derived equivalent over their field of definition, it is not known 
exactly when there is a
derived equivalence between
the generalized Kummer fourfolds
$K_2(A)$ and $K_2(\hat{A})$. Recently, it was shown that, over an algebraically closed field of characteristic zero, they are derived equivalent when $A$ has a polarization of exponent coprime to $3$ \cite[Theorem 1]{Magni}.

Given these results, we ask the following two questions, which we examine in Sections \ref{sec:r} and \ref{sec:de}, respectively.

\begin{question}\label{qrouquier}
Suppose we have a derived equivalence of abelian surfaces $D^b(A)\cong D^b(B)$.  
How do the groups $G_A(v)$ introduced in Section \ref{sec:involutions}
interact with the
Rouquier isomorphism $A\times\hat{A}\simeq B\times\hat{B}$?
\end{question}  

\begin{question}\label{whende}
Under what conditions are irreducible symplectic fourfolds of Kummer type derived equivalent?
\end{question}

Throughout this section, we will assume we are working with data as in Setting~\ref{defKv} and that all varieties $K_A(v)$ are an Albanese fiber over symmetric line bundles.

\subsection{Compatibility with the Rouquier isomorphism}\label{sec:r}

\begin{prop}[{{Rouquier, cf.~\cite[Prop.~9.45]{HuyFMbook}}}]
  Let $A$ and $B$ be abelian varieties and $F\colon D(A) \to D(B)$ a derived equivalence.
There is an isomorphism
$f \colon A\times \hat{A} \to B \times \hat{B}$, called the Rouquier isomorphism, which maps $(a,\alpha)\in A\times \hat{A}$ to the unique element $(b,\beta)\in B\times \hat{B}$ so that the following diagram commutes:
\begin{equation}\label{R}
\vcenter{  
\xymatrix{
 D(A) \ar[r]^{F} \ar[d]_{L_\alpha \otimes t_a^*} & D(B) \ar[d]^{L_\beta \otimes t_b^*}
  \\
D(A) \ar[r]^{F} & D(B).  
}
}
\end{equation} 
\end{prop}

The following proposition 
gives some results addressing Question~\ref{qrouquier}.

\begin{prop}\label{prop:rouquier}
Let $A$ and $B$ be abelian surfaces over a field $k$, 
and let $v=(r,l,s)\in N(A)$ and $v'=(r',l',s')\in N(B)$.

Let $F\colon D(A)\to D(B)$ be a derived equivalence 
such that $F(v)=v'$.
Then the base change of the
Rouquier isomorphism to the algebraic closure $\bar{k}$
restricts to a group scheme isomorphism
\begin{equation}\label{fkbar}
  f_{\bar{k}}\colon G_{A_\kbar}(v)\xrightarrow{\sim} G_{B_\kbar}(v')
\end{equation}  
under any of the following conditions:
\begin{enumerate}[{\rm (a)}]
\item For any elements $\fF,\gG\in M_A(v)$ such that $\alb(\fF)=\alb(\gG)$,
  we have $\det(F(\fF))=\det(F(\gG))$ and $\det(\Phi_P\circ F(\fF))=\det(\Phi_P\circ F(\gG))$;
\item $F$ is a stability-preserving Fourier--Mukai transform; that is, if $E\in M_A(v)$, then $F(E)$ is in $M_B(v')$; or
\item $k=\C$ and $\frac{v^2}{2}= 3$ (i.e.~$K_A(v)$ is a fourfold).
\end{enumerate}
\end{prop}

We note that the isomorphism \eqref{fkbar} 
implies that
the actions of $\Gal(\bar{k}/k)$ on $G_{A_\kbar}(v)$ and $G_{B_\kbar}(v')$
are isomorphic.

\begin{proof}
Let $(a,\alpha)\in G_{A_\kbar}(v)$. 
By Remark~\ref{respectalllinebundles},
to prove that $(b,\beta):=f_{\bar{k}}(a,\alpha)\in G_{B_{\bar{k}}}(v')$, it suffices to produce an element $\hH\in D(B)$ where $v(\hH)=v'$,
$\det(\hH)=\det(L_\beta\otimes t_b^*\hH)$, and
$\det(\Phi_P(\hH))=\det(\Phi_P(L_\beta\otimes t_b^*\hH))$.

Under condition (a), for any $\fF\in M_A(v)$, we may take 
$\hH:=F(\fF)$. In this case we have
$L_\beta\otimes t_b^*\hH=F(L_{\alpha}\otimes t_a^*\fF)$. Since
$$
\det(\fF) = \det( L_{\alpha}\otimes t_a^*\fF)
\quad\text{and}\quad
\det(\Phi_P(\fF)) = \det(\Phi_P( L_{\alpha}\otimes t_a^*\fF)),
$$
condition (a) allows us to conclude that $\hH$ has the needed property.

Under condition (b), $F$ restricts to an isomorphism $M_A(v)\to M_B(v')$ and by the universal property of the Albanese morphism there is a commutative diagram as follows:
\[
\xymatrix{
  M_A(v) \ar[d]^{\alb}\ar[r]^{F} & M_B(v')\ar[d]^{\alb}\\
  \Pic^l(A)\times \Pic^{\lambda}(\hat{A})\ar@{-->}[r]^{\exists}
  &\Pic^{l'}(A)\times \Pic^{\lambda'}(\hat{A})
}
\]  
Thus $F$ satisfies condition (a).

By \cite[Prop.~4.3]{HasTsc} if $K_A(v)$ is a fourfold, the intersection of the fixed loci of $\kappa$ and $(L_{\alpha}\otimes t_a^*)\kappa$ acting on $K_A(r,l,s)$ is nonempty. For instance, in $K_2(A)$ the intersection of
$\Fix(\kappa)$ and $\Fix(\kappa_{(\tau,0)})$ where $\tau\in A[3]$ (cf.~Lemma \ref{translation}) 
contains $(0,\tau,-\tau)$.

Let $\gG$ be an element in this intersection. It is thus fixed by $L_{\alpha}\otimes t_a^*$. Following the diagram above, we see that $\hH:=F(\gG)$ is fixed by $L_{\beta}\otimes t_b^*$ and thus $F$ satisfies the needed condition.
\end{proof}

\begin{rmk}
The barrier to a proof of Proposition~\ref{prop:rouquier} under more general conditions
  is that it is not known that a general Fourier--Mukai equivalence
 will respect the Albanese morphism acting on $M_A(v)$.

The proof of Proposition~\ref{prop:rouquier} under condition (c) hinges on the selection of an element fixed by the automorphisms from Theorem~\ref{n4}. We anticipate that analogous results are available
for higher-dimensional varieties of Kummer type.
For instance, in $K_{n-1}(A)$ the intersection between
$\Fix(\iota_{(0,0)})$ and $\Fix(\iota_{(\tau,0)})$ where $\tau\in A[n]$ 
contains $(0,\tau,2\tau,\ldots,(n-1)\tau)$. 
\end{rmk}

\begin{example}\label{FMgroups}
(a) For any abelian surface $A$ we have the Fourier--Mukai equivalence
$\Phi_P:D(A)\to D(\hat{A})$.
For any Mukai vector $v$ on $A$, condition (a) of Proposition~\ref{prop:rouquier} is satisfied for $F=\Phi_P$ since $\Phi_P\circ\Phi_P=\iota^*\circ [-2]$. 
If $v:=(r,l,s)$, then $v':=F(v)=(s,m,r)$ \cite[Lemma~3.1]{Yoshioka}, and $G_{A_\kbar}(v)$ and $G_{\hat{A}_\kbar}(v')$
are very closely related via the canonical identification between an abelian surface and the dual of its dual. By Theorem~\ref{n4}, the elements in $G_{A_\kbar}(v)$ satisfy the equations shown in \eqref{LM} and the elements of $G_{\hat{A}_\kbar}(v')$ satisfy the equations
\[
\phi_m(y)=-sx,\quad \phi_l(x)=ry
\quad\text{for } (y,x)\in \hat{A}\times\hat{\hat{A}}
\]
Thus $(x,y)\in G_{A_\kbar}(v)$ if and only if $(-y,x)\in G_{\hat{A}_\kbar}(v')$.

\smallskip

\noindent(b) Let $A$ be an abelian surface defined over a field $k$ of characteristic $0$ with $\NS(A)=\Z l$ and $l^2=2n$. By \cite[Lem.~3.6]{Gulbrandsen}
the Fourier--Mukai equivalence $L\otimes^{\mathbb{L}}(-)\colon D(A) \to D(A)$
satisfies condition (a) of Proposition~\ref{prop:rouquier}; in fact
$M_A(1,0,-n) \cong M_A(1,l,0)$. Moreover, by \cite[Prop.~3.5]{Yoshioka},
applying the Fourier--Mukai transform $\Phi_P$ followed by a shift $[-1]$ gives an isomorphism
$M_{A}(1,l,0)\cong M_{\hat{A}}(0,-\hat{l},-1)$, where $\hat{l}$ is the N\'eron-Severi class of $\Phi_P(1,l,0)$. If $l$ is an ample generator of $\NS(A)$, then $-\hat{l}$ is an ample generator of $\NS(\hat{A})$.

The shift functor $[1]$ acts on Mukai vectors by multiplication by $-1$, and in general $G_A(v)=G_A(-v)$. 
Thus, there are isomorphisms of group schemes
$$G_A(1,0,-n)\cong G_{A}(1,{l},0)
\cong G_{\hat{A}}(0,-\hat{l},-1),$$
though 
as discussed in Remark~\ref{gpexample},
the groups 
$G_A(1,0,-n)$ and $G_{A}(1,{l},0)$ are distinct subgroups of $(A\times \hat{A})[n]$. 
\end{example}  

\subsection{Derived equivalence of fourfolds of Kummer type}
\label{sec:de}

The following result provides some information on 
Question~\ref{whende} and allows us to produce an example where two such varieties over a number field $k$ are not derived equivalent over $k$.

\begin{prop}\label{decohomology}
Let $A$ and $B$ be isogenous abelian surfaces over a finitely generated field $k$ of characteristic $0$.
Let $v$ and $v'$ be Mukai vectors with $v^2=v'^2=6$, so that $K_A(v)$ and $K_B(v')$ are fourfolds. 
If $K_A(v)$ and $K_B(v')$ are derived equivalent over $k$,
then $\Ql[G_{A_\kbar}(v)]$ and
$\Ql[G_{B_\kbar}(v')]$ are 
isomorphic as $\Gal(\bar{k}/k)$-representations.
\end{prop}

We begin with a lemma about the orthogonal complement to $v$ in the Mukai lattice.

\begin{lemma}\label{lem:vperp}
Let $A$ be an abelian surface over a field $k$ and $v$ a Mukai vector with $v^2\geq 2$. Let $v^\perp\subset \tilde{H}(A_{\kbar}, \Ql)$ be the orthogonal complement to $v$ under the Mukai pairing. Then there is a Galois equivariant isomorphism $v^\perp \cong H^2_{\et}(A_\kbar,\Ql(1))\oplus \Ql$.
\end{lemma}

\begin{proof}
Let $w:=(1,0,-n)$ for $n:=\frac{v^2}{2}\geq 1$, and note that $$w^\perp = H^2_{\et}(A_\kbar,\Ql(1))\oplus \Ql\langle (1,0,n)\rangle.$$ We will show that $v^\perp \cong w^\perp$. For any $y \in \tilde{H}(A_\kbar, \Ql)$ with $y^2\neq 0$, let reflection through $y$ be given by
\[x \mapsto x - \frac{2\langle x,y\rangle}{y^2} y.\]
Observe that $(v-w)^2\neq 0$ or $(v+w)^2 \neq 0$, and so reflection through $v-w$ or $v+w$ gives an isometry $\tilde{H}(A_\kbar, \Ql) \xrightarrow{\sim} \tilde{H}(A_\kbar, \Ql)$ which sends $v$ to $\pm w$. Thus the isometry restricts to a Galois equivariant isomorphism $v^\perp \xrightarrow{\sim} w^\perp$.
\end{proof}

\begin{lem}
\label{lem:H2} 
Let $A$ be an abelian surface over a field $k$ and $v$ a Mukai vector with $v^2\geq 6$. 
Then there is a Galois equivariant isomorphism 
\[H^2_{\et}(K_A(v)_{\kbar},\Ql(1)) \cong H^2_{\et}(A_\kbar,\Ql(1))\oplus \Ql.\]
\end{lem}

\begin{proof}
By \cite[Thm.~0.2(2)]{Yoshioka}, along with the comparison theorem for singular and \'etale cohomology and the smooth base change theorem, we have a Galois equivariant isomorphism $H^2_{\et}(K_A(v)_{\bar{k}},\Q_\ell(1))\cong v^\perp$ (In fact, this isomorphism exists over $\Z_\ell$, while the isomorphism of Lemma \ref{lem:vperp} may only exist over $\Q_\ell$). This combined with Lemma~\ref{lem:vperp} gives the result.
\end{proof}

\begin{proof}[{{Proof of Proposition \ref{decohomology}}}] 
Suppose that $K_A(v)$ and $K_B(v')$ are derived equivalent, so they 
have isomorphic sums of even cohomologies after Tate twists \cite[Lem.~3.1]{Honigs}: $\widetilde{H}(K_A(v)_\kbar,\Ql)\cong \widetilde{H}(K_B(v')_\kbar,\Ql)$.
We know that the zeroth and top cohomologies of $K_A(v)$ and $K_B(v')$ are trivial Galois representations, and Lemma~\ref{lem:H2} gives that 
\[H^2_{\et}(K_A(v)_\kbar, \Ql(1))\cong H^2_{\et}(A_\kbar, \Ql(1))\oplus \Ql.\]
By Theorem~\ref{thm:middlecohom} and Poinc\'are duality (cf.~\cite{Honigs3}), it follows that there is an isomorphism of Galois modules
\[\widetilde{H}(K_A(v),\Ql)\cong \Ql^{\oplus 4} \oplus H^2_{\et}(A_\kbar, \Ql(1))^{\oplus 2} \oplus \Sym^2H^2_{\et}(K_A(v)_{\bar{k}},\Q_\ell(1)) \oplus V_A,\]
where $V_A:=V$ from  Theorem~\ref{thm:middlecohom}. There is a similar isomorphism for $\widetilde{H}(K_B(v'),\Ql)$ involving $V_B$. 
We will check that these representations are semisimple, so that we can reduce to a comparison of $V_A$ and $V_B$.

By \cite[Thm.~3]{Faltings} and its extension to finitely generated fields of characteristic~0 in \cite[Thm.~4.3]{Zarkhin},
$H^2_{\et}(A_\kbar, \Ql)$ is a semisimple representation,  
%\khedit{By the same arguments as in the proof of Proposition~\ref{charp}, we have that
and thus so is $\Sym^2H^2_{\et}(K_A(v)_{\bar{k}},\Q_\ell(1))$. The $\Gal(\kbar/k)$-representation $\Ql[G_{A_\kbar}(v)]$ factors through a finite group representation, determined by the finite extension of $k$ over which $G_{A_\kbar}(v)$ is defined, and so by Maschke's theorem it is also semisimple. Thus, the representation $\widetilde{H}(K_A(v),\Ql)$ is semisimple.
The same also holds for $\widetilde{H}(K_B(v'),\Ql)$, so applying Schur's Lemma, this allows us to cancel isomorphic representations in the direct sums for $\widetilde{H}(K_A(v),\Ql)$ and for $\widetilde{H}(K_B(v'),\Ql)$.
Since $A$ and $B$ are isogenous, there is an isomorphism $H^2_{\et}(A_{\bar{k}},\Q_\ell)\cong H^2_{\et}(B_{\bar{k}},\Q_\ell)$,  
so along with the observations above,
we are reduced to an isomorphism $V_A\cong V_B$. This extends to an isomorphism $\Ql[G_{A_\kbar}(v)] \cong \Ql[G_{B_\kbar}(v')]$, as desired.
\end{proof}

We use this result to give a negative answer to Question~\ref{whende}
in the case of generalized Kummer varieties
$K_2(A)$ and $K_2(\hat{A})$. 

%\begin{cor}\label{cor:notderived}
%There exists an abelian surface $A$ defined over a number field $k$ for which 
%$K_2(A)\cong K_A(1,0,-3)$ and $K_2(\hat{A})\cong K_A(3,0,-1)$ are not derived equivalent over $k$. 
%\end{cor}
\begin{cor}\label{cor:notderived}
For an abelian surface $A$ defined over a number field $k$ for which $\Ql[A[3]]$ and $\Ql[\hat{A}[3]]$ are not isomorphic as Galois modules over $k$, $K_2(A)\cong K_A(1,0,-3)$ and $K_2(\hat{A})\cong K_A(3,0,-1)$ are not derived equivalent over $k$. 
\end{cor}

\begin{proof}
We have $G_{A_\kbar}(1,0,-3)=A[3]$
and
by the discussion in  Example~\ref{FMgroups}(a),
$G_{A_\kbar}(3,0,-1)=G_{\hat{A}_\kbar}(1,0,-3)=\hat{A}[3]$.
The result then follows by Proposition \ref{decohomology}.
\end{proof}  

In \cite{FHV} the authors exhibit an abelian surface $A$ defined over a number field $k$ where $\Ql[A[3]]$ and $\Ql[\hat{A}[3]]$ are not isomorphic as Galois modules over $k$.

\begin{rmk}
If $A$ is an abelian surface as in the proof of
  Corollary~\ref{cor:notderived},
any derived equivalence between $K_2(A)$ and $K_2(\hat{A})$
would have to be defined over a field larger than $k$.
Moreover, the kernel of such a derived equivalence could not be constructed out of only universal bundles, since such bundles would naturally be defined over $k$, and the derived equivalence would descend. 
\end{rmk}

\begin{rmk}
The argument in Corollary~\ref{cor:notderived} cannot be used to rule out
derived equivalences between $K_2(A)$ and $K_2(\hat{A})$ in many contexts; for instance it does not work when $A$ is principally polarized, since such a polarization would give an isomorphism between $A[3]$ and $\hat{A}[3]$.

Proposition~\ref{decohomology} also holds for 
Kummer varieties over fields of positive characteristic
  that satisfy the hypotheses of Proposition~\ref{charp};
Tate's theorem gives the needed semisimplicity result \cite{TateThm}.
However, over a finite field in general, Tate's isogeny theorem implies there is an isomorphism between the Tate modules $T_{\ell}A$ and $T_{\ell}\hat{A}$. Thus it would not be possible to use the approach
of Corollary~\ref{cor:notderived}
to rule out a 
derived equivalence between $K_2(A)$ and $K_2(\hat{A})$ if $A$ were defined over a finite field.
\end{rmk}

%% file: JacobiansK3.tex
In this and the following sections,
we consider an extended example where
we work over $\C$. 

Let $(A,L)$ be a polarized abelian surface where $L$ 
is symmetric, $\NS(A)=\Z l$ for $l:=c_1(L)$ and $l^2=6$, so 
$L$ is a $(1,3)$-polarization (see Claim~\ref{kernel}).
Let $K_A(0,l,s)$ be as in Setting~\ref{defKv}, and assume
$M\in \Pic^{m}(\hat{A})$ is also symmetric. 
We will see below that the spaces $K_A(0,l,s)$ are fibered over $\P^2$ in Jacobians of irreducible genus 4 curves, and while they can be identified fiberwise as $s$ varies, their global geometry differs: the discriminant of the Beauville--Bogomolov--Fujiki form on $\Pic(K_A(0,l,s))$ changes, so these moduli spaces are not in general birational.

We consider the fixed locus of $K_A(0,l,s)$ under the action of $\iota^*$, which we refer to as $\Fix(\iota^*)$. By Lemma~\ref{translation}, the fixed locus of any symplectic involution on $K_A(0,l,s)$ is a translation of $\Fix(\iota^*)$.
The moduli space $K_A(0,l,s)$
parametrizes
rank $1$ stable sheaves, or equivalently,
rank $1$ torsion-free sheaves, 
supported on irreducible curves in $A$.
When the supporting curves are smooth, these sheaves are line bundles on the curves, but we also encounter curves with nodal singularities, in which case the space of rank 1 torsion-free sheaves naturally compactifies the space of line bundles. 

In this section, we give necessary background and show that there is a natural fibration of 
$K_A(0,l,s)$ in abelian surfaces such that
$\Fix({\iota^*})\subset K_A(0,l,s)$ contains an elliptically fibered K3 surface.
In Section~\ref{sec:Jacobianssingular}, we will analyze 
the singular fibers of this K3 surface,
and in Section~\ref{sec:JacobiansIsolatedpts} we will analyze the isolated points of the fixed locus.

For comparison, we first give a description of $\Fix(\iota^*)$ in $K_A(1,0,-3)$ here.

\subsection{$\Fix(\iota^*)$ for $K_2(A)$}
\label{isolated}

The points in $K_2(A)\cong K_A(1,0,-3)$ consist of $0$-dimensional length $3$ subschemes of $A$ for which the support sums to~$0$. 
It was shown in \cite[Thm.~4.4]{HasTsc} that $\Fix(\iota^*)$ contains the
Kummer K3 surface
\begin{equation}\label{fixkummer}
\overline{\{(a_1,a_2,a_3)\mid a_1=0, a_2=-a_3, a_2\neq 0\}}
\end{equation}
as well as a unique isolated point
supported at the identity element $0$.

Any length $3$ subscheme in $\Fix(\iota^*)$ 
containing a point $a\in A$ in its support that is not fixed by $\iota^*$ must
be of the form $(0,a,-a)$, 
which is in the Kummer K3 surface described above. 
Thus, the remaining isolated points in
$\Fix(\iota^*)$ found by Tar\'i \cite{Tari} 
must consist of 
triples of three distinct points of $A[2]\cong (\Z/2\Z)^4$ that sum to~$0$.
The identity element cannot be contained in such a triple. 
Once we have chosen two of the points the third is forced, and length $3$ subschemes are unordered, 
so we have
\[
\tfrac13\tbinom{15}{2} =35
\]
such isolated points.

\subsection{Stable sheaves and compactifications of the Jacobian}\label{Abelmap}

Let $\Pic^d(C)$ be the set of degree $d$ line bundles on any curve $C$.
We write $\Pic^d_C$ for the Picard scheme of degree $d$ on a
curve $C$, and we use 
$\Picbar^d_C$ to denote the moduli scheme parametrizing rank $1$ degree $d$
torsion-free sheaves on the mildly singular curves $C$ that arise in this paper, which are all Gorenstein and moreover have planar singularities.

If $C$ is elliptic, $\Pic^d_C\cong C$ for any $d$, and this fact has some generalizations to compactified Jacobians of singular genus $1$ curves that we will find useful.

\begin{prop}[{{\cite[\S3, p.~14]{Kass},\cite[Ex.~39]{Esteves}}}]\label{genusone}
Let $C$ be a genus~$1$ reduced curve that is irreducible and nodal. Then  $\Picbar^d_C\cong C$  for any $d$. 
\end{prop}

The Abel map \cite[Def.~1.0.5]{Kass} 
and a generalization of it
for compactified Jacobians of Gorenstein curves is
useful to our arguments. We use the development of this map given by Kass in \cite{Kass}, though we do not need the full power of Kass's theory. 

\textit{Generalized divisors} on $C$ are nonzero subsheaves
of the sheaf of the total quotient ring of $C$,
$I_D\subset \mathcal{K}$,
that are coherent $\O_C$-modules. These divisors generalize Cartier divisors, which they coincide with when $I_D$ is a line bundle. 
An \textit{effective generalized divisor} on $C$ is a $0$-dimensional closed subscheme $Z\subset C$, meaning the following generalization of the Abel map continues to have the intuitive quality of sending points to corresponding elements in $\Picbar^{-d}_C$ \cite[Def.~5.0.7]{Kass}, \cite[Thm.~8.5]{AK}:
\begin{align}
\label{abelmap}
\alpha\colon \Hilb^d_C &\to \Picbar^{-d}_C\\
[D] &\mapsto I_D\notag
\end{align}
When the degree $d$ is greater than or equal to the arithmetic genus $g$, this map is
surjective and generically has fibers isomorphic to $\P^{d-g}$.
If $D$ is an effective generalized divisor, $\alpha^{-1}([I_D])$ is the complete linear system $|D|$.
If $g=d$, the map is generically injective. The locus where $\alpha$ is non-injective in this case is 
\textit{the exceptional locus} $C^1_d$,
which consists of divisors $D$ whose image under the canonical map lies on a hyperplane. Such divisors $D,D'$ are linearly equivalent if there are canonical divisors $K,K'$ such that $K-D=K'-D'$.

Related to Proposition~\ref{genusone},
this generalized Abel map is an isomorphism when $C$ is a nodal genus-$1$ curve.

\subsection{The Lagrangian fibration of $K_A(0,l,s)$}
Since $l^2=6$, $K_A(0,l,s)$ is $4$-dimensional and,
since $\NS(A)=\Z l$ for $l:=c_1(L)$,
the curves $C\in|L|$ are irreducible, hence all rank 1 torsion-free sheaves are stable. Thus, $K_A(0,l,s)$
parametrizes rank 1 torsion-free sheaves on irreducible
curves $C\subset A$ where $C\in|L|$, which are generically line bundles.
Curves in this linear system have arithmetic genus~$4$ 
and by Riemann-Roch, the line bundles parametrized by $K_A(0,l,s)$
have degree $d:=s+3$.

We see that $h^0(A,L)=3$, and $h^1(A,L)=h^2(A,L)=0$.
Thus there is a map sending elements of $K_A(0,l,s)$ to their supports in the linear system~$|L|$:
\begin{align}\label{f}
f\colon K_A(0,l,s)&\to \P^2\cong |L|\\
\fF&\mapsto \supp(\fF) \notag
\end{align}

\begin{lemma}\label{phiC}
Let $C\in |L|$ and $h_C\colon C\hookrightarrow A$ be the natural inclusion.
The fiber of $f$ over $C\in |L|$
is the fiber over $M$ of the following surjective morphism:
\begin{align}\label{phiCdef}
\varphi_C\colon \Picbar_C^d &\to \Pic^{m}_{\hat{A}}\\
\fF &\mapsto \det(\Phi_P(h_{C*}\fF)).\notag
\end{align}
This fiber $f^{-1}(C)=\varphi_C^{-1}(M)$ is a translation of the
fiber of the following map over $0_A$:
\begin{align}\label{universal}
j_C:\Picbar^0_C\to A\\
I_D\to -\Sigma D,\notag
\end{align}
where $\Sigma D$ is the sum of points in the divisor $D$ using the group law on $A$.
\end{lemma}

When $C$ is smooth, $j_C$ 
is
the morphism given by the universal property of the Jacobian, which sends a line bundle, e.g. $\O(p-q)$, to $p-q$.

\begin{rmk}\label{jC}
  We will analyze the $\iota^*$-invariant portion of $\varphi_C^{-1}(M)$ in later results. This lemma shows that we may reduce to analyzing the
$\iota^*$-invariant portion of the 
fiber of $j_C$ over $0_A$, which we call $\ker j_C$, somewhat abusing notation in the singular case.
\end{rmk}

\begin{proof}
Recall that $K_A(0,l,s)$ is the fiber of the Albanese map over $(L,M)$:
\begin{equation}\label{alb.ex}
\alb\colon M_A(0,l,s)\to \Pic^l_A\times \Pic^{m}_{\hat{A}}.
\end{equation}
We consider the interaction of $\alb$ with $f$.
Let $\cC$ be the tautological family of curves in $|L|$.
We may identify the fiber of \eqref{alb.ex}
over $\{L\}\times \Pic^{m}_{\hat{A}}$ 
with the relative compactified Jacobian $\smash{\Picbar^d_{\cC/\P^2}}$,
which also has a map to supports $g:\smash{\Picbar^d_{\cC/\P^2}}\to |L|$.
Thus there is an inclusion
$K_A(0,l,s) \hookrightarrow \smash{\Picbar^d_{\cC/\P^2}}$ making the following diagram commute:
\begin{equation}\label{inclusion}
\vcenter{  
\xymatrix{
K_A(0,l,s) \ar@{^{(}->}[r] \ar[dr]_{f} & \Picbar^d_{\mathcal{C}/\P^2} \ar[d]_{g} \\
& |L|\cong \P^2. 
}
}
\end{equation}
For any curve $C\in|L|$, the fiber of $g$ over $C$ is $\Picbar^d_C$, which is isomorphic to $\Pic^d_C$ if $C$ is smooth. 
The morphism $\varphi_C$ given in the statement of the lemma is the restriction of the Albanese morphism \eqref{alb.ex} on $M_A(0,l,s)$ 
to $\Picbar^d_C$. 
Using \eqref{inclusion}, we see the fiber 
of $f$ over $C$ is equal to the fiber of
$\varphi_C$ over $M$.

Let $\L$ be a line bundle on $C$ and $p$  a point in $C$.
As in  \cite[\S17.2]{Polishchukbook}, applying  $\det(\Phi_P(h_{C*}-))$ to the short exact sequence
$$0\to \L \to  \L\otimes \O(p)\to k(p)\to 0 $$
implies
$$\varphi_C(\L\otimes \O(p))=\varphi_C(\L)\otimes P_{p}$$
where $P_p$ 
is the line bundle on $\hat{A}$ corresponding to $p\in C\subset A$.
Moreover, for any divisor $D$ on $C$, we have
$$\varphi_C(\L\otimes \O(D))=\varphi_C(\L)\otimes P_{\Sigma D},$$
where $P_{\Sigma D}$ is the line bundle on $\hat{A}$ corresponding to the point on $A$ that comes from summing $D$ using the group law on $A$.
If $C$ is singular, this argument may be extended to ideal sheaves of generalized divisors $D$.
Thus $\varphi_C$ is a translation of
the morphism $j_C$ of \eqref{universal}
by an element of $\Picbar^d(C)$.

If $C$ is smooth, 
the map induced by applying the universal property of the Jacobian
to the inclusion $C\hookrightarrow A$ is surjective \cite{BSgenus4curves},
thus $\varphi_C$ is surjective as well.

 Alternately, to prove $\varphi_C$ is surjective for smooth curves $C$, we may observe that $\varphi_C$ is equivariant under the action of $\Pic^0(A)$ and the action of $\Pic^0(A)$ on $\Pic^{m}\hat{A}$ is transitive.
For singular $C$, if we restrict $\varphi_C$ to $\Pic^d_C\subset \Picbar^d_C$, the same argument holds and so $\varphi_C$ is surjective.
\end{proof}

In \cite{Gulbrandsen}, Gulbrandsen shows that the map $f\colon K_A(0,l,-1)\to \P^2$
is a Lagrangian fibration. There is a similar Lagrangian fibration of $K_A(0,l,s)$ for any choice of $s$.

\begin{prop}\label{absurffib}
For any $s$, the map $f\colon K(0,l,s)\to \P^2$ is a Lagrangian fibration.
\end{prop}  

\begin{proof}
By \cite[Thm.~1]{Matsushita}, it suffices to prove that $f$ is surjective  and its fibers are connected. By Lemma~\ref{phiC}, the fiber of $f$ over $C\in |L|$ is the fiber of $\varphi_C$ over $M$, which is non-empty since $\varphi_C$ is surjective. Thus $f$ is surjective. By \cite[Lem.~2.6]{BSgenus4curves}, the fibers of $f$ over smooth curves are connected.
By considering the Stein factorization of $f$, we conclude that $f$ has connected fibers.
\end{proof}  

\subsection{The Lagrangian fibration restricted to $\Fix(\iota^*)$}

Since $K(0,l,s)$ is fibered over $|L|$, we begin by analyzing the action of $\iota^*$ on $|L|$.

The restriction of the Weil pairing $\langle-,\phi_L(-)\rangle$ on points in $A$ to $A[2]$ yields a quadratic form $q_L\colon A[2]\to \mu_2$. 
Since $\ker(\phi_L)\cong (\Z/3\Z)^2$ (see Claim~\ref{kernel}),
it contains only the trivial element of $A[2]$, hence 
$q_L$ is nondegenerate.
Whether $q_L$ is even or odd 
as a quadratic form
(cf.~\cite[p.~63]{Polishchukbook}, \cite[\S3]{BolognesiMassarenti})
determines several facts about the action of $\iota^*$ on $|L|$. 

\begin{prop}\label{evenodd}
The action of $\iota^*$ on $H^0(A,L)$ decomposes into eigenspaces
$H^0(A,L)_+$ and $H^0(A,L)_-$
with eigenvalues $\pm 1$. Furthermore:
\[
\dim(H^0(A,L)_+)=
\begin{cases}
2 & \text{if $q_L$ is even} \\
1 & \text{if $q_L$ is odd} 
\end{cases}
\quad
\dim(H^0(A,L)_-)=
\begin{cases}
1 & \text{if $q_L$ is even} \\
2 & \text{if $q_L$ is odd} 
\end{cases} 
\]
We call the $1$-dimensional and $2$-dimensional eigenspaces, respectively,
$$V_{\hyp} \quad\text{and}\quad V_{\el}.$$
For generic $A$,
any curve $C\in \P V_{\hyp}$  is smooth and hyperelliptic, and 
there are $10$ points in $A[2]$ though which it passes. If $q_L$ is even, then $0_A$ is among these $10$ points.
The remaining $6$ points in $A[2]$ are the base locus of $\P V_{\el}$. 
If $q_L$ is odd, then $0_A$ is among these $6$ points.
\end{prop}

The space $V_{\hyp}$ was named for the fact the curves in it are hyperelliptic. The name $V_{\el}$ was chosen because, by Riemann--Hurwitz, quotients $C/\iota$ of smooth curves $C\in \P V_{\el}$ are elliptic.

\begin{proof}
Calculations on the dimensions of $H^0(A,L)_{\pm}$ and
the number of points through which these curves pass have been carried out in \cite[\S3]{BolognesiMassarenti}, \cite[\S3]{BSgenus4curves} and \cite{Naruki}.
See
\cite[Ch.~4]{BirkenhakeLange} and \cite[Ch.~13]{Polishchukbook} for further details. 
\end{proof}

By Proposition~\ref{K3pts}, $\Fix(\iota^*)$ consists of a K3 surface and $36$ isolated points. Here we study the geometry of the K3 surface.

\begin{prop}\label{ellfib}
The K3 surface in   $\Fix(\iota^*)$
is elliptically fibered. \end{prop}

\begin{proof}
By Proposition~\ref{evenodd},
$$|L|^{\iota^*}=\P V_{\el} \sqcup \P V_{\hyp} \cong \P^1\sqcup \P^0,$$
and thus $\Fix(\iota^*)$
is fibered over $\P^1\sqcup \P^0$.

By Lemma~\ref{phiC}, the fiber of $K_A(0,l,s)$ over $C$ is the fiber of $\varphi_C$ over $M$.
Let $C\in |L|^{\iota^*}$ be smooth.
By Remark~\ref{jC}, to determine the dimension of the $\iota^*$-invariant parts of this fiber, 
we examine the eigenvalues of the action of $\iota^*$ on the tangent space of $\ker j_C$.

We have a short exact sequence on tangent spaces
\[0 \to T_0\ker j_C \to T_0\Pic^0_C \to T_0 A \to 0.\]
The tangent space $T_0A$ is $H^1(A,\O_A)$, and it is $2$-dimensional with $\iota^*$ acting as multiplication by $-1$. The tangent space $T_0\Pic^0_C$ is $H^1(C,\O_C) \cong H^0(C,\omega_C)^*$,
which is $4$-dimensional. 
On the other hand, tensoring 
the short exact sequence
\[
0\to\O_A(-C)\to\O_A\to\O_C\to 0.
\]
with $L\cong \O_A(C)$ gives
\[
0\to\O_A\to L\to\O_C(C)\to 0.
\]
By adjunction, $\O_C(C)\cong \omega_C$, so we have the following long exact sequence:
\[
0\to H^0(A,\O_A)\to H^0(A,L)\to H^0(C,\omega_C)\to H^1(A,\O_A)\to 0.
\]
The map $H^0(A,\O_A)\to H^0(A,L)$ sends the generator of the $1$-dimensional space $H^0(A,\O_A)$ to $[C]$. 
Putting all of this together,
the eigenvalues and dimensions of eigenspaces of $\iota^*$ acting on the tangent space of $\ker j_C$ are equal to those of $\iota^*$ acting on $H^0(A,L)/[C]$.

Suppose $C\in \P V_{\hyp}$. Then by Proposition~\ref{evenodd} the eigenvalues of $\iota^*$ acting on  $H^0(A,L)/[C]$ are both the same: 
if $q_L$ is even, they  are both $+1$ and if $q_L$ is odd,
they are both $-1$. In each case these eigenvalues are different from the eigenvalue of the action of $\iota^*$ on $[C]$.
If instead $C\in \P V_{\el}$, then for $q_L$ even or odd the eigenvalues of $\iota^*$ acting on  $H^0(A,L)/[C]$  are $+1$ and $-1$.

The tangent space of the fiber of $\Fix(\iota^*)$ over
$C\in \P V_{\hyp}\sqcup \P V_{\el}$ is isomorphic to the eigenspace
of $\iota^*$ acting on $T_0\ker j_C$ with the same eigenvalue as the action of $\iota^*$ on $[C]$.
Thus, $\Fix(\iota^*)$ has $0$-dimensional fibers over $\P V_{\hyp}$
and generically $1$-dimensional fibers over $\P V_{\el} \cong\P^1$.
For any $C\in \P V_{\el}$ that is smooth, the fiber of $j_C$
over $0_A$
is $2$-dimensional and so must be an abelian surface. Since $\iota^*$ acts with two different eigenvalues on the tangent space of $\ker j_C$, it must be, up to isogeny, the product of two elliptic curves. 

Generically, curves $C\in \P V_{\el}$ are smooth,
and as mentioned above,
$C/\iota$ is an elliptic curve.
Since the quotient map $C\to C/\iota$ is a ramified cyclic double cover mapping between smooth varieties, pullback 
induces an
inclusion $\Pic^0(C/\iota)\hookrightarrow \Pic^0(C)$.
We may represent any point in the image as $\O(x+\iota(x))$ for some $x\in C$. Such line bundles are in $\ker j_C$. Similarly, there is an inclusion of
$\smash{\Pic^d_{C/\iota}}$ into $(\varphi_C^{-1}(M))^{\iota^*}$, and by the tangent space calculation we see that generically these elliptic curves $\Pic^d_{C/\iota}\cong C/\iota$ are the $1$-dimensional part of the fiber of $f$ over $C$.
\end{proof}

In the case of $K_A(0,l,-1)$, we are able to give the following refinement by a different argument.

\begin{prop}\label{fixedKummer}
The fixed locus of $\iota^*$ on $K_A(0,l,-1)$ consists of the Kummer K3 surface $K_1(A)\cong K_1(\hat{A})$ and 36 isolated points.
\end{prop}

\begin{proof}
Hassett and Tschinkel \cite{HasTsc} and Tar\'i \cite{Tari} showed that the fixed locus of a symplectic involution on $K_2(A)$ consists of the Kummer K3 surface $K_1(A)$ and 36 additional isolated points.

As discussed in Example~\ref{FMgroups}(b), a series of derived equivalences
compatible with $\iota^*$
gives an isomorphism $K_A(0,l,-1)\cong K_{\hat{A}}(1,0,-3)$. 
Hence the K3 surface in the fixed locus of $\iota^*$ acting on $K_A(0,l,-1)$ is isomorphic to $K_1(\hat{A})$, which is isomorphic over $\C$ to $K_1(A)$ \cite{HLOY,Stellari}.
\end{proof}

%% file: Jacobianssingular.tex
The proof of Proposition~\ref{ellfib} shows that the fibration $K_A(0,l,s) \to |L|$ restricts to a fibration $\Fix(\iota^*) \to \P V_{\el}$, and when $C\in \P V_{\el}$ is smooth, the fiber of $\Fix(\iota^*)$ over $C$ is isomorphic to $\Pic^d_{C/\iota}$.
It remains to examine the fibers in $\Fix(\iota^*)$ over
curves in $\P V_{\el}$
that are singular.
We show below that the singular fibers are the same as the singular fibers of a natural elliptic fibration of the Kummer K3 of $A$, which we now describe.

In \cite{Naruki}, Naruki analyzes an elliptic fibration of Kummer K3 surfaces that are constructed from $(1,3)$-polarized abelian surfaces. 
He uses the linear system $\P V_{\el}$ of Proposition~\ref{evenodd} to induce a linear system we will call $W$ on $K_1(A)$, which yields an elliptic fibration $K_1(A) \to \P^1$ whose fibers are generically $C/\iota$ for $C\in \P V_{\el}$.
Since $C\in \P V_{\el}$ must have arithmetic genus $4$ and pass through at least $6$ points in $A[2]$,
Riemann--Hurwitz
shows that
if $C/\iota$ is a smooth elliptic curve, then $C$ must be smooth as well and pass through exactly $6$ points in $A[2]$.

\begin{prop}[Naruki {{\cite[\S4]{Naruki}}}]\label{propnaruki}
Under a genericity assumption on $A$ \cite[p.~224, (GA)]{Naruki},
  the linear system $W$ has:
\begin{enumerate}[{\rm (i)}]  
\item Four singular fibers of type $I_1$.
\item Ten singular fibers of type $I_2$. There is one fiber of this type for each point of $A[2]$ that is not in the base locus of $\P V_{\el}$. The line in $K_1(A)$ that is the blow up of this point is contained in the fiber.
\end{enumerate}
\end{prop}

We show that the same is true for $\Fix(\iota^*)$:

\begin{thm}\label{thm:badfibers}
Let $A$ be an abelian surface satisfying the hypotheses at the beginning of the section
such that the singular fibers of $W$ consist of four fibers of type $I_1$ and ten fibers of type $I_2$ as in Proposition~\ref{propnaruki}. 
Then, for any $s$, $\Fix(\iota^*)\subset K_A(0,l,s)$ contains an elliptically fibered K3
whose singular fibers are of the same type.
\end{thm}

\begin{proof}
We split the proof into two parts. In Proposition~\ref{I1} below, we show that there are 4 fibers of type $I_1$. In Proposition~\ref{I2}, we show that there are 10 fibers of type $I_2$.

For topological reasons, this must be all of the one-dimensional locus of $\Fix(\iota^*)\subset K_A(0,l,s)$. Indeed, the $4$ singular fibers of type $I_1$ and $10$ singular fibers of type $I_2$ account for the fact that the topological Euler number of a K3 surface is $24$ \cite[Rmk.~11.1.12]{HuyK3book}.
\end{proof}

\begin{prop}\label{I1}
  Let $C\in \P V_{\el}$ be a curve inducing a genus $1$ singular curve $C/\iota$ of type $I_1$ in $W$. Then $\smash{\Picbar^0_{C/\iota}\cong \Picbar^d_{C/\iota}}$ is a singular curve of type $I_1$ and includes into $(\ker j_C)^{\iota^*}$.
\end{prop}  

\begin{proof}
By assumption, 
the curve $C/\iota$ is of type $I_1$, hence has arithmetic genus $1$ with one nodal singularity.
Applying the Riemann--Hurwitz formula for singular curves
\cite[(1.2)]{GarciaLax} 
to the double cover
$C\to C/\iota$, we see the arithmetic genus $4$ curve $C$ has geometric genus $2$ with $6$ ramification points,
so it must have two singular points that are exchanged by $\iota$. We call these points $x$ and $\iota x$
and then write $[x,\iota x]$ for the singular point of $C/\iota$.

Consider the induced map on the normalizations of these curves: $C^{\nu}\to (C/\iota)^{\nu}$. Since this is a ramified double cover of curves, the pullback map
$\Pic^0((C/\iota)^{\nu})\to \Pic^0(C^{\nu})$ 
is an inclusion.
We have the following map between short exact sequences of groups:
\[
\xymatrix{
0 \ar[r]& \C^*\ar[r]\ar@{^{(}->}[d] &\Pic^0(C/\iota)\ar[r]\ar[d] & \Pic^0((C/\iota)^{\nu}) \ar[r]\ar@{^{(}->}[d] & 0
\\  
0\ar[r] & \C^*\oplus \C^*\ar[r] &\Pic^0(C)\ar[r]& \Pic^0(C^{\nu})\ar[r]& 0
}
\]
The elements of $\C^*$ correspond to all possible choices for identifying the two fibers over a given node.
The vertical maps are pullbacks along quotient maps and $\C^*\to \C^*\oplus \C^*$ is the diagonal map, which corresponds to a choice of gluing a line bundle at the node on $C/\iota$ getting mapped to the same choice of gluing at each of the nodes on $C$.
By the five lemma, the map $\Pic^0(C/\iota)\to \Pic^0(C)$ is an injection.

The Abel map (see Section~\ref{Abelmap})
shows the points in
$\smash{\Picbar^{-1}_{C/\iota}}$
correspond to points on the curve $C/\iota$, and all its elements are line bundles except for the sheaf corresponding to the singular point of $C/\iota$.

The pullback map
$\smash{\Picbar^{-1}_{C/\iota}\to \Picbar^{-2}_{C}}$ sends $\L([x,\iota x])$ to $\L([x]+[\iota x])$, which is also not a line bundle and maps to $0$ under $\varphi_C$.
The pullback map $\smash{\Picbar^{-1}_{C/\iota}\to \Picbar^{-2}_{C}}$ is thus an injection.
The sheaf $\L([x,\iota x])$ is fixed by $\iota^*$.
We may choose an isomorphism
$\smash{\Picbar^{-1}_{C/\iota}\cong \Picbar^0_{C/\iota}}$
compatible with $\iota^*$ to see that
$\smash{\Picbar^0_{C/\iota}}$
includes into $(\ker{j_C})^{\iota^*}$.

By Proposition~\ref{genusone}, $\Picbar^0_{C/\iota}$ is a singular curve of type $I_1$.
\end{proof}

\begin{prop}\label{I2}
Let $C\in \P V_{\el}$ be a curve inducing a genus $1$ singular curve
$X$ of type $I_2$ in $W$. Then $(\ker {j_C})^{\iota^*}$ contains  a curve of type $I_2$. 
\end{prop}

\begin{proof}
By the discussion in \cite{Naruki}, the
curve $X$
in the linear system $W$ that corresponds to $C$
is the intersection of a line and a conic.
The line in $X$ is the blow-up of a point $q\in A[2]$ that is one of the $10$ such not in the base locus of $\P V_{\el}$. 
The curve $C$ thus has a node at $q$.
The normalization $f \colon C^{\nu} \to C$
inherits an action of $\iota$,
and the quotient $C^{\nu}/\iota$ is
the conic contained in $X$.
Thus, $C^{\nu}$ is hyperelliptic and
as a double cover of $C^{\nu}/\iota$
it is ramified at $8$ points, consisting of the six points $p_1,\ldots,p_6$ in the base locus of $\P V_{\el}$ and the two 
points above $q$, call them $q_1,q_2$.
By Riemann--Hurwitz, $C^{\nu}$ has genus~$3$.
Thus $C$ has arithmetic genus $4$  and geometric genus $3$, so 
the node at $p$ is its unique singularity.

Via Altman and Kleiman's presentation schemes \cite{AKpres}, 
we have the following description of $\Picbar^0_C$ (cf.~\cite[\S3.3]{Kass2008LectureNO}).
Pullback by the normalization map $f:C^{\nu}\to C$ gives the short exact sequence on Picard groups
\begin{equation}\label{normalize}
0\to \C^* \to \Pic^0(C) \to \Pic^0(C^{\nu}) \to 0.
\end{equation}
where again the elements of $\C^*$ correspond to all possible choices for identifying the two fibers over $q$.
The presentation scheme of $f$ gives 
a $\P^1$-bundle $\pi \colon P\to \Pic^0_{C^\nu}$, 
where the fiber over a point $I'\in \Pic^0_{C^\nu}$
is given by presentations of $I'$, that is, short exact sequences
of sheaves on $C$ of the following form:
\[
0\to I\to f_*I'\to k(q)\to 0.
\]  
It follows that $I\in \Picbar^0_C$, 
so there is a natural
morphism $\kappa \colon P \to \Picbar^0_C$,
which is an isomorphism when restricted to the preimage of $\Pic^0_C\subset \Picbar^0_C$.
For each $I' \in \Pic^0_{C^\nu}$,
there is a $\C^* \subset \P^1=\pi^{-1}(I')$,
exactly the $\C^*$ of \eqref{normalize},
which gets mapped injectively under $\kappa$ into $\Pic^0_C$.

Furthermore, there is a closed embedding $\varepsilon'\colon \Pic^0_{C^\nu} \times \{q_1,q_2\} \hookrightarrow P$, which sends a pair $(I',q_i)$ to the presentation
\[0 \to f_*I'(-q_i) \to f_*I' \to f_*(I'|_{q_i}) \to 0.\] This gives the description of the rest of the $\P^1$-fiber of $\pi$ over a point $I' \in \Pic^0_{C^\nu}$: these are the two points compactifying the $\C^*$ described above. Thus to complete the description of $\Picbar^0_C$, it remains to describe $\kappa$ restricted to $\varepsilon'(\Pic^0_{C^\nu} \times \{q_1,q_2\})$. 

Here, $\kappa$ is $2$-to-$1$, but does not just trivially glue the two copies of $\Pic^0_{C^\nu}$ together. Rather they are glued with a twist: 
\[\kappa\varepsilon'(I',q_1)=\kappa\varepsilon'(I'(q_1-q_2),q_2).\]
Since $C^\nu$ is hyperelliptic, $2q_1 \sim_{\text{lin}} 2q_2$ and
$\O(q_1-q_2)$ is $2$-torsion in $\Pic^0_{C^\nu}$,
which further implies that
\[\kappa\varepsilon'(I'(q_1-q_2),q_1)=\kappa\varepsilon'(I',q_2).\]

With this observation in hand, we now describe the one-dimensional component of the locus of $\Picbar^0_C$ that is in $(\ker j_C)^{\iota^*}$. While we are working in $\Picbar^0_C$, we will instead consider the fiber over $0_A$ of $\varphi_C$. By abuse of notation, we will also call this $\ker \varphi_C$. 

\begin{claim}
The locus of $\Picbar^0_C$ that is both fixed by $\iota^*$ and is in $\ker \varphi_C$ contains
\[\kappa\left(\pi^{-1}(\O_{C^\nu})\cup \pi^{-1}(\O_{C^\nu}(q_1-q_2))\right),\]
which is two copies of $\P^1$ intersecting at two points, i.e.~a singular curve of type $I_2$.
\end{claim}

\begin{proof}
  First we observe that if $I'\in \Pic^0_{C^{\nu}}$ is fixed by $\iota^*$,
and $$0\to I\to f_*I'\to k(q)\to 0$$ is a presentation of $I'$, then 
$I$ is also fixed by $\iota^*$.
Indeed, by push-pull, we know that $f_*\iota^*I'\cong \iota^*f_*I'$.
We have the short exact sequence
\[
0\to \iota^*I\to \iota^*f_*I'\to \iota^*k(q)\to 0,
\]
so if $I'$ is fixed by $\iota^*$, then $I\cong \iota^*I$. Thus, if $I'$ is fixed, then the whole $\P^1$-fiber in $P$ is pointwise fixed as well.
Note also that, given a short exact sequence
$0\to I\to f_*I'\to k(q)\to 0$, if
$I\in \ker \varphi_C$, then 
by the discussion in the proof of Lemma~\ref{phiC} about the behavior of $\varphi_C$ in short exact sequences, 
so are all the other possible $I$ giving presentations of $I'$. 

Since there is a short exact sequence
\[
0\to \O_C\to f_*\O_{C^\nu}\to k(q)\to 0
\]
and $\O_C\in \ker\varphi_C$, 
we know any other
kernels of presentations of $f_*\O_{C^\nu}$ 
will as well.
We also have that $\O_{C^\nu}$ is fixed by $\iota^*$, so it follows that $\kappa(\pi^{-1}(\O_{C^\nu}))$ is both fixed by $\iota^*$ and
in $\ker\varphi_C$. 

The same holds for $\kappa(\pi^{-1}(\O_{C^\nu}(q_1-q_2))$: since $q_1$ and $q_2$ are fixed by $\iota$, $\O_{C^\nu}(q_1-q_2)$ is fixed by $\iota^*$,
and we will show that 
any presentation is sent to $0$ by $\varphi_C$.
For a presentation
\[0 \to I \to f_*\O_{C^\nu}(q_1-q_2) \to k(q) \to 0,\]
applying the inclusion $h_C:C\hookrightarrow A$ and $\Phi_P$,
we have
\[\det \Phi_P(h_{C*}I) \otimes \mathcal{P}_q\cong \det \Phi_P(h_{C*}f_*\O_{C^\nu}(q_1-q_2)). \]
There is also a presentation
\begin{equation}\label{pres1}
  0 \to f_*\O_{C^\nu}(-q_1) \to f_*\O_{C^\nu}(q_1-q_2) \to k(q) \to 0,
  \end{equation}
which gives 
\[\det \Phi_P(h_{C*}f_*\O_{C^\nu}(-q_1)) \otimes \mathcal{P}_q\cong \det \Phi_P(h_{C*}f_*\O_{C^\nu}(q_1-q_2)), \]
and hence
\[\det \Phi_P(h_{C*}I) \cong \det \Phi_P(h_{C*}f_*\O_{C^\nu}(-q_1)).\]
On the other hand, there is a presentation
\begin{equation}\label{pres2}
  0 \to f_*\O_{C^\nu}(-q_1) \to f_*\O_{C^\nu} \to k(q) \to 0,
\end{equation}
so $f_*\O_{C^\nu}(-q_1)\in \kappa(\pi^{-1}(\O_{C^\nu}))$, which we showed above
is in $\ker\varphi_C$. 
Thus the same is true for $I$.

It remains to show that these two $\P^1$'s in $\Picbar^0_C$ are glued together at two points. But this follows from the description of $\Picbar^0_C$, since
\[\kappa\varepsilon'(\O_{C^\nu},q_1)=\kappa\varepsilon'(\O_{C^\nu}(q_1-q_2),q_2),\]
and
\[\kappa\varepsilon'(\O_{C^\nu}(q_1-q_2),q_1)=\kappa\varepsilon'(\O_{C^\nu},q_2).\qedhere\]
\end{proof}

\noindent While above we work in degree $0$, we can twist by a degree $d$ line bundle on $C$ to get the description in $\Picbar^d_C$. 
This completes the proof of Proposition~\ref{I2}.
\end{proof}

\begin{rmk}\label{rmk:compare}
It is interesting to consider $(\ker j_C)^{\iota^*}$ in Propostion~\ref{I2}
from the point of view of the Abel map. We consider the case where $d=-4$. The fixed locus
$(\ker j_C)^{\iota^*}\subseteq \Picbar_C^{-4}$
contains all divisors of the form $-(x+\iota x+ y+\iota y)$ for $x,y\in C$, but
the information from the 
Abel map alone does not
make clear which of these divisors get identified under linear equivalence in $\Picbar_C^{-4}$.
We will show that these divisors are all contained in the curve from Proposition~\ref{I2}. 
We may choose an isomorphism $\Picbar^0_C\cong\Picbar^{-4}_C$ by subtracting four copies of a $2$-torsion point $p\in C$ not at the node. Let $p'\in C^{\nu}$ be preimage of $p$ under the normalization map. Since $C^\nu$ is hyperelliptic, $\O_{C^\nu}(-4p')\cong \O_{C^\nu}(-x'-\iota x'-y'-\iota y')$, where $x'$ and $y'$ are the preimages of $x$ and $y$ in $C^{\nu}$. There is a presentation
\[0 \to \O_C(-x-\iota x-y-\iota y) \to (f_*\O_{C^\nu})(-4p)\to k(q) \to 0,\]
so these divisors all lie in the $\P^1$ corresponding to the twist of $\kappa(\pi^{-1}(\O_{C^\nu}))$.

The divisors $-(x+\iota x+ y+\iota y)$
correspond to a two-dimensional family in $\Hilb^4_C$, but their image in
$\Picbar^{-4}_C$ is at most $1$-dimensional, so they must
lie in the exceptional locus of the Abel map. Since $C$ is not hyperelliptic, the canonical morphism gives a closed immersion into $\P^3$ and divisors in the exceptional locus are those that lie on a hyperplane in $\P^3$; we see there must be an 
interaction of these planes with the action of $\iota$, but the particulars of it are not immediately clear.

It would also be nice to have a description
of the elements in $\Picbar^{-4}_C$ contained in the other copy of $\P^1$ in $(\ker j_C)^{\iota^*}$.
Since the canonical bundle $\omega_C$ is fixed by $\iota^*$ and 
$f^*\omega_C\cong \omega_{C^{\nu}}\otimes\O_{C^{\nu}}(q_1+q_2)\cong \O_{C^{\nu}}(4p'+q_1+q_2)$, line bundles on $C$ which fit into presentations with middle term $f_*\left(f^*\omega_C^{-1}\otimes \O_{C^{\nu}}(x'+\iota x')\right)$ 
lie in the $\P^1$ corresponding to the twist of 
$\kappa(\pi^{-1}(\O_{C^\nu}(q_1-q_2))$. However, the question of exactly which effective divisors give rise to these line bundles is again dependent on the
geometry of the canonical embedding.
\end{rmk}

\begin{rmk}
In this section we have shown that the elliptic K3 surface in $\Fix(\iota^*)$
has the same types of singular fibers as those
in the fibration of the Kummer K3 surface
studied by Naruki~\cite{Naruki}.
By Proposition~\ref{fixedKummer} for 
$K_A(0,l,-1)$, the K3 surface in the fixed-point locus is isomorphic to the Kummer K3 surface.
However, it is not apparent that in general there is any kind of natural map from the K3 surface studied by Naruki to $\Fix(\iota^*)$, or that these fixed-point loci are Kummer K3 surfaces.
\end{rmk}

%% file: JacobiansIsolatedpts.tex
Finally, we seek a description of the $36$ isolated points in $\Fix(\iota^*)$. We will use a combination of the Abel map and information about the
geometry of $2$-torsion points in a $(1,3)$-polarized abelian surface 
to finish our description of the fixed loci.

\subsection{Geometry of $A[2]$}

The description of the isolated points in $\Fix(\iota^*)\subset K_A(0,l,s)$ will require an understanding of line bundles on curves $C \in |L|^{\iota^*}$ corresponding to divisors which sum to $0$ in $A$.

As discussed in \cite{BarthNieto}, 
the line bundle $L^2$ on our $(1,3)$-polarized abelian surface gives an embedding of the \textit{desingularized} Kummer K3 surface into $\P^3$.
They describe an action of the Heisenberg group on $\P^3$ that connects the geometry of the group action of elements $A[2]$ to the corresponding lines in the Kummer K3 surface.

We use notation from
Hudson's analysis of $A[2]$ for principally polarized abelian surfaces \cite[Ch.~1,\S 4]{Hudson}, which has the same group structure:
We write the group of points of $A[2]$ in multiplicative notation
in terms of (a not minimal set of) generators $1,A,B,C,A',B',C'$ 
where $1$ is the identity.
The following multiplication tables hold:
\[
    \begin{tabular}{>{$}l<{$}|*{6}{>{$}l<{$}}}
    ~ &  A  & B & C  \\
    \hline\vrule height 12pt width 0pt
    A &  1 & C   & B \\
    B &     & 1   & A \\
    C &    &   & 1 \\
    \end{tabular}
\quad\quad\quad
    \begin{tabular}{>{$}l<{$}|*{6}{>{$}l<{$}}}
    ~ &  A'  & B' & C'  \\
    \hline\vrule height 12pt width 0pt
    A' &  1 & C'   & B' \\
    B' &     & 1   & A' \\
    C' &    &   & 1 \\
    \end{tabular}
\]
Following Naruki \cite{Naruki}, we see 
the six points of $A[2]$ that occur in the base locus of $\P V_{\el}$
must be a set of six
in the group that coincides with those
that would lie on a plane in
Hudson's $(16,6)$ configuration, and so we
take the following six points to be in the base locus of $\P V_{\el}$:
\begin{equation}\label{six}
AB', AC', BC', BA', CA', CB'
\end{equation}
Any possible choice of six points will have the same numerical properties
described below 
as they will differ by a translation.

The remaining ten points of $A[2]$ are then:
\begin{equation}\label{ten}
1,A, A', B, B', C, C',AA',BB',CC'
\end{equation}

We will need the following observations in our 
identification of the isolated fixed points:

\begin{lemma}\label{count}
\begin{enumerate}[{\rm (a)}]  
\item The product of any four distinct points in \eqref{six} cannot be the identity.
\item Given any point in \eqref{ten}, there are exactly two ways to then choose three distinct points from those in \eqref{six} so that the product of the four points is the identity. 
\item There are fifteen ways to choose four distinct points from among \eqref{ten} so that their product is the identity.
\end{enumerate}  
\end{lemma}

\begin{proof}
The results may be verified directly.

In part (b), for instance, if we choose $1$, we have exactly
\[
(1)(AB')(BC')(CA') \quad\text{and}\quad (1)(AC')(BA')(CB').
\]
  
In part (c), the fifteen possibilities are:
\begin{align*}
&(1)(A)(A')(AA') & &(A)(B)(C')(CC') \\
&(1)(B)(B')(BB') & &(A')(B')(C)(CC') \\
&(1)(C)(C')(CC') & &(A)(B')(C)(BB') \\
&(1)(AA')(BB')(CC') & &(A')(B)(C')(BB') \\
&(1)(A)(B)(C) & &(A')(B)(C)(AA') \\
&(1)(A')(B')(C') & &(A)(B')(C')(AA')\\
&(A)(A')(BB')(CC') & &(B)(B')(AA')(CC')  \\
& & &(C)(C')(AA')(BB') \qedhere
\end{align*}
\end{proof}

\subsection{The fiber of  $\Fix(\iota^*)$ over $\P V_{\hyp}$}\label{abelfix}

Let $C\in \P V_{\hyp}\cong \P^0$. For $A$ a general $(1,3)$-polarized abelian surface, $C$ is smooth by \cite[Lem.~3.4]{BSgenus4curves},
so the kernel of $\varphi_C$ \eqref{phiCdef}
is an abelian surface (see Prop.~\ref{absurffib}). 
The action $\ker\varphi_C$ inherits from $\iota^*$ on $K(0,l,s)$ is the action of $[-1]$ on it as an abelian surface.
Thus there will be exactly $16$ isolated fixed points, consisting of the $2$-torsion points on $\ker\varphi_C$.

We may also analyze $(\ker\varphi_C)^{\iota^*}$
using the Abel map $\alpha$ from~\eqref{abelmap}.
Since $C$ is hyperelliptic, the canonical morphism is the degree $2$ morphism $\pi\colon C\to \P^1$. 
The canonical divisors of $C$ are
of the form $\pi^{-1}(t_1)+\pi^{-1}(t_2)+\pi^{-1}(t_3)$ for $t_1,t_2,t_3\in\P^1$.
The sets of points in $\Hilb^{4}_C$ which sum to $0$ and are fixed by $\iota^*$ 
consist of points of the form $\pi^{-1}(t_1)+\pi^{-1}(t_2)$, which are all linearly equivalent, and of four distinct $2$-torsion points that sum to~$0$.

From this point of view we find that the sixteen isolated points in 
$\Pic^{-4}_C$ that sum to $0$ and are fixed by $\iota^*$ are (the negative of) the fifteen points given by Lemma~\ref{count}(c) and the one point that is the image under $\alpha$ of all points of the form $\pi^{-1}(t_1)+\pi^{-1}(t_2)$.

This argument may be used to show the same result holds for
$\Pic^d_C$. 
We can take the 
isomorphism $\Pic^{-4}_C\cong\Pic^d_C$ to be given by adding $d+4$ copies of a fixed $2$-torsion point $p$, 
in which case the isomorphism commutes with $\iota^*$.
It is not always possible to choose this isomorphism
so that it commutes with taking the kernel of the summation map, but we may instead consider the elements in $\Pic^d_C$ that sum to $(d+4)\cdot p$, which amounts to simply performing this calculation in a different
fiber of the Albanese map \eqref{albtor}, which is related to our preferred fiber by an isomorphism.

\subsection{The fibers of  $\Fix(\iota^*)$ over $\P V_{\el}$}

In the last section we found 16 of the 36 isolated points in $\Fix(\iota^*)$.
To find the rest we examine $(\ker\varphi_C)^{\iota^*}$ as $C$ varies in $\P V_{\el}\cong \P^1$. 

If $C$ is smooth, by our analysis of the tangent space of
$\ker\varphi_C$ in the proof of Proposition~\ref{ellfib}, 
$(\ker\varphi_C)^{\iota^*}$ is isomorphic to the elliptic curve $C/\iota$, and there are no isolated points.

Let $C\in \P V_{\el}$ be singular of type $I_1$.
Consider the modified Abel map
$$\alpha\colon \Hilb^4_C\to \Picbar^{-4}_C.$$
The curve $C$ passes through exactly six $2$-torsion points \eqref{six} (cf.~Proposition \ref{I1}).
By Lemma~\ref{count}(a)
the only sets of four points on $C$ that sum to $0$ and are fixed by $\iota^*$ are those of the form $(x,\iota x,y,\iota y)$ for some $x,y\in C$.
Points of this form are already contained in the image of the pullback
$\Pic_{C/\iota}^0\cong \Pic_{C/\iota}^{-2}\to \Pic_{C}^{-4}$. 
Thus $(\ker\varphi_C)^{\iota^*}\cong C/\iota$, and there are no isolated points.

Now let $C\in \P V_{\el}$ be singular of type $I_2$.
The curve $C$ passes through seven points in $A[2]$: those in the
base locus of $\P V_{\el}$ (see~\eqref{six})
as well as one additional $2$-torsion point.
The points in $\Hilb^4_C$ that sum to $0$ and are fixed by $\iota^*$ are those of the form $(x,\iota x,y,\iota y)$ for some $x,y\in C$, as well as
any tuple of four $2$-torsion points that sum to $0$.
By Lemma~\ref{count}(a,b), there are exactly two points of the latter form and they are isolated from the points of the former form (cf.~Remark~\ref{rmk:compare}).
Thus the fiber of $\Fix(\iota^*)$ over $C$ consists precisely of a singular curve of type $I_2$ and two isolated fixed points.
There are ten such singular curves, and thus all of the $36$ isolated points in $\Fix(\iota^*)$ are now accounted for.

\begin{rmk}
It would be interesting to use the presentation scheme description of $\Picbar^d_C$, for $C\in \P V_{\el}$ singular, to identify the two isolated points in the fiber of $\Fix(\iota^*)$ over $C$. For example, what line bundles do they pull back to on $C^\nu$?
\end{rmk}

%% file: main.bbl
\begin{thebibliography}{BNWS11}

\bibitem[AK80]{AK}
A.~Altman and S.~Kleiman.
\newblock Compactifying the {P}icard scheme.
\newblock {\em Adv. in Math.}, 35(1):50--112, 1980.

\bibitem[AK90]{AKpres}
A.~Altman and S.~Kleiman.
\newblock The presentation functor and the compactified {J}acobian.
\newblock In {\em The {G}rothendieck {F}estschrift, {V}ol. {I}}, volume~86 of
  {\em Progr. Math.}, pages 15--32. Birkh\"{a}user Boston, Boston, MA, 1990.

\bibitem[Bec21]{Beckmann}
T.~Beckmann.
\newblock Derived categories of hyper-k\"ahler manifolds: extended {M}ukai
  vector and integral structure.
\newblock {\em arXiv preprint arXiv:2103.13382}, 2021.

\bibitem[BL03]{BLseparable}
Ch. Birkenhake and H.~Lange.
\newblock An isomorphism between moduli spaces of abelian varieties.
\newblock {\em Math. Nachr.}, 253:3--7, 2003.

\bibitem[BL04]{BirkenhakeLange}
Ch. Birkenhake and H.~Lange.
\newblock {\em Complex abelian varieties}, volume 302 of {\em Grundlehren der
  Mathematischen Wissenschaften [Fundamental Principles of Mathematical
  Sciences]}.
\newblock Springer-Verlag, Berlin, second edition, 2004.

\bibitem[BM16]{BolognesiMassarenti}
M.~Bolognesi and A.~Massarenti.
\newblock Moduli of abelian surfaces, symmetric theta structures and theta
  characteristics.
\newblock {\em Comment. Math. Helv.}, 91(3):563--608, 2016.

\bibitem[BN94]{BarthNieto}
W.~Barth and I.~Nieto.
\newblock Abelian surfaces of type {$(1,3)$} and quartic surfaces with {$16$}
  skew lines.
\newblock {\em J. Algebraic Geom.}, 3(2):173--222, 1994.

\bibitem[BNWS11]{HigherEn}
S.~Boissi\`ere, M.~Nieper-Wi{\ss}kirchen, and A.~Sarti.
\newblock Higher dimensional {E}nriques varieties and automorphisms of
  generalized {K}ummer varieties.
\newblock {\em J. Math. Pures Appl. (9)}, 95(5):553--563, 2011.

\bibitem[BS17]{BSgenus4curves}
P.~Bor\'{o}wka and G.~K. Sankaran.
\newblock Hyperelliptic genus 4 curves on abelian surfaces.
\newblock {\em Proc. Amer. Math. Soc.}, 145(12):5023--5034, 2017.

\bibitem[Don69]{Donovan}
P.~Donovan.
\newblock The {L}efschetz-{R}iemann-{R}och formula.
\newblock {\em Bull. Soc. Math. France}, 97:257--273, 1969.

\bibitem[Est01]{Esteves}
E.~Esteves.
\newblock Compactifying the relative {J}acobian over families of reduced
  curves.
\newblock {\em Trans. Amer. Math. Soc.}, 353(8):3045--3095, 2001.

\bibitem[Fal83]{Faltings}
G.~Faltings.
\newblock Endlichkeitss\"{a}tze f\"{u}r abelsche {V}ariet\"{a}ten \"{u}ber
  {Z}ahlk\"{o}rpern.
\newblock {\em Invent. Math.}, 73(3):349--366, 1983.

\bibitem[FHVon]{FHV}
S.~Frei, K.~Honigs, and J.~Voight.
\newblock On abelian varieties whose torsion is not self-dual.
\newblock In preparation.

\bibitem[FL21]{FuLi}
L.~Fu and Z.~Li.
\newblock Supersingular irreducible symplectic varieties.
\newblock In {\em Rationality of varieties}, volume 342 of {\em Progr. Math.},
  pages 147--200. Birkh\"{a}user/Springer, Cham, 2021.

\bibitem[Fog73]{Fogarty}
J.~Fogarty.
\newblock Fixed point schemes.
\newblock {\em Amer. J. Math.}, 95:35--51, 1973.

\bibitem[Fre20]{Frei}
S.~Frei.
\newblock Moduli spaces of sheaves on {K}3 surfaces and {G}alois
  representations.
\newblock {\em Selecta Math. (N.S.)}, 26(1):Paper No. 6, 16, 2020.

\bibitem[GKLR22]{LLVdecomp}
M.~Green, Y.~Kim, R.~Laza, and C.~Robles.
\newblock The {LLV} decomposition of hyper-{K}\"{a}hler cohomology (the known
  cases and the general conjectural behavior).
\newblock {\em Math. Ann.}, 382(3-4):1517--1590, 2022.

\bibitem[GL96]{GarciaLax}
A.~Garc\'{\i}a and R.~F. Lax.
\newblock Rational nodal curves with no smooth {W}eierstrass points.
\newblock {\em Proc. Amer. Math. Soc.}, 124(2):407--413, 1996.

\bibitem[Gul06]{Gulbrandsen}
M.~Gulbrandsen.
\newblock {\em Fibrations on generalized {K}ummer varieties}.
\newblock PhD thesis, University of Oslo, 2006.

\bibitem[HL10]{HL}
D.~Huybrechts and M.~Lehn.
\newblock {\em The geometry of moduli spaces of sheaves}.
\newblock Cambridge Mathematical Library. Cambridge University Press,
  Cambridge, second edition, 2010.

\bibitem[HLOY03]{HLOY}
S.~Hosono, B.~Lian, K.~Oguiso, and S.~Yau.
\newblock Kummer structures on {$K3$} surface: an old question of {T}.
  {S}hioda.
\newblock {\em Duke Math. J.}, 120(3):635--647, 2003.

\bibitem[Hon15]{Honigs}
K.~Honigs.
\newblock Derived equivalent surfaces and abelian varieties, and their zeta
  functions.
\newblock {\em Proc. Amer. Math. Soc.}, 143(10):4161--4166, 2015.

\bibitem[Hon18]{Honigs3}
K.~Honigs.
\newblock Derived equivalence, {A}lbanese varieties, and the zeta functions of
  3-dimensional varieties.
\newblock {\em Proc. Amer. Math. Soc.}, 146(3):1005--1013, 2018.
\newblock With an appendix by J. D. Achter, S. Casalaina-Martin, K. Honigs, and
  C. Vial.

\bibitem[HT13]{HasTsc}
B.~Hassett and Y.~Tschinkel.
\newblock Hodge theory and {L}agrangian planes on generalized {K}ummer
  fourfolds.
\newblock {\em Mosc. Math. J.}, 13(1):33--56, 189, 2013.

\bibitem[Hud90]{Hudson}
R.~W. H.~T. Hudson.
\newblock {\em Kummer's quartic surface}.
\newblock Cambridge Mathematical Library. Cambridge University Press,
  Cambridge, 1990.
\newblock With a foreword by W. Barth, Revised reprint of the 1905 original.

\bibitem[Huy06]{HuyFMbook}
D.~Huybrechts.
\newblock {\em Fourier--{M}ukai transforms in algebraic geometry}.
\newblock Oxford Mathematical Monographs. The Clarendon Press, Oxford
  University Press, Oxford, 2006.

\bibitem[Huy16]{HuyK3book}
D.~Huybrechts.
\newblock {\em Lectures on {K}3 surfaces}, volume 158 of {\em Cambridge Studies
  in Advanced Mathematics}.
\newblock Cambridge University Press, Cambridge, 2016.

\bibitem[Kas08]{Kass2008LectureNO}
J.~Kass.
\newblock Lecture notes on compactified {J}acobians.
\newblock 2008.
\newblock {\scriptsize https://people.math.sc.edu/kassj/Lecture Notes on
  Compactified Jacobians.pdf}.

\bibitem[Kas13]{Kass}
J.~Kass.
\newblock Singular curves and their compactified {J}acobians.
\newblock In {\em A celebration of algebraic geometry}, volume~18 of {\em Clay
  Math. Proc.}, pages 391--427. Amer. Math. Soc., Providence, RI, 2013.

\bibitem[KLS06]{KLS}
D.~Kaledin, M.~Lehn, and Ch. Sorger.
\newblock Singular symplectic moduli spaces.
\newblock {\em Invent. Math.}, 164(3):591--614, 2006.

\bibitem[KM18]{KapferMenet}
S.~Kapfer and G.~Menet.
\newblock Integral cohomology of the generalized {K}ummer fourfold.
\newblock {\em Algebr. Geom.}, 5(5):523--567, 2018.

\bibitem[KMO22]{KMO}
Ljudmila Kamenova, Giovanni Mongardi, and Alexei Oblomkov.
\newblock Symplectic involutions of {$K3^{[n]}$} type and {K}ummer {$n$} type
  manifolds.
\newblock {\em Bull. Lond. Math. Soc.}, 54(3):894--909, 2022.

\bibitem[Lan04]{La}
A.~Langer.
\newblock Semistable sheaves in positive characteristic.
\newblock {\em Annals of Mathematics (2)}, 159(1):251--276, 2004.

\bibitem[LZ21]{twistLiZou}
Z.~Li and H.~Zou.
\newblock Derived isogenies and isogenies for abelian surfaces.
\newblock {\em arXiv preprint arXiv:2108.08710}, 2021.

\bibitem[LZ23]{LiZou}
Z.~Li and H.~Zou.
\newblock A note on {F}ourier-{M}ukai partners of abelian varieties over
  positive characteristic fields.
\newblock {\em Kyoto J. Math., to appear}, 2023.
\newblock See also arXiv preprint arXiv:2107.05404.

\bibitem[Mag22]{Magni}
P.~Magni.
\newblock Derived equivalences of generalized {K}ummer varieties.
\newblock {\em arXiv preprint arXiv:2208.11183}, 2022.

\bibitem[Mar22]{MaKummers}
E.~Markman.
\newblock The monodromy of generalized {K}ummer varieties and algebraic cycles
  on their intermediate {J}acobians.
\newblock {\em J. Eur. Math. Soc.}, 2022.
\newblock Published online first.

\bibitem[Mat01]{Matsushita}
D.~Matsushita.
\newblock Addendum: ``{O}n fibre space structures of a projective irreducible
  symplectic manifold'' [{T}opology {\bf 38} (1999), no. 1, 79--83; {MR}1644091
  (99f:14054)].
\newblock {\em Topology}, 40(2):431--432, 2001.

\bibitem[Mil16]{Milnebook}
J.~Milne.
\newblock {\em Etale cohomology (PMS-33)}, volume~33.
\newblock Princeton University Press, 2016.

\bibitem[Muk81]{Mukai}
S.~Mukai.
\newblock Duality between {$D(X)$} and {$D(\hat X)$} with its application to
  {P}icard sheaves.
\newblock {\em Nagoya Math. J.}, 81:153--175, 1981.

\bibitem[Muk84]{Mukaisymp}
S.~Mukai.
\newblock Symplectic structure of the moduli space of sheaves on an abelian or
  {K3} surface.
\newblock {\em Inventiones mathematicae}, 77(1):101--116, 1984.

\bibitem[Muk87a]{MukaiFF}
S.~Mukai.
\newblock Fourier functor and its application to the moduli of bundles on an
  abelian variety.
\newblock In {\em Algebraic geometry, {S}endai, 1985}, volume~10 of {\em Adv.
  Stud. Pure Math.}, pages 515--550. North-Holland, Amsterdam, 1987.

\bibitem[Muk87b]{Mukaiexposition}
S.~Mukai.
\newblock Moduli of vector bundles on {K3} surfaces and symplectic manifolds.
\newblock {\em S\={u}gaku}, 39(3):216--235, 1987.
\newblock Sugaku Expositions {{\bf{1}}} (1988), no. 2, 139--174.

\bibitem[Mum70]{Mumford}
D.~Mumford.
\newblock {\em Abelian varieties}, volume~5 of {\em Tata Institute of
  Fundamental Research Studies in Mathematics}.
\newblock Published for the Tata Institute of Fundamental Research, Bombay by
  Oxford University Press, London, 1970.

\bibitem[MW15]{MW}
G.~Mongardi and M.~Wandel.
\newblock Induced automorphisms on irreducible symplectic manifolds.
\newblock {\em J. Lond. Math. Soc. (2)}, 92(1):123--143, 2015.

\bibitem[MW17]{autoOG}
G.~Mongardi and M.~Wandel.
\newblock Automorphisms of {O}'{G}rady's manifolds acting trivially on
  cohomology.
\newblock {\em Algebr. Geom.}, 4(1):104--119, 2017.

\bibitem[Nar91]{Naruki}
I.~Naruki.
\newblock On smooth quartic embedding of {K}ummer surfaces.
\newblock {\em Proc. Japan Acad. Ser. A Math. Sci.}, 67(7):223--225, 1991.

\bibitem[O'G99]{OG10}
K.~O'Grady.
\newblock Desingularized moduli spaces of sheaves on a {K3}.
\newblock {\em J. Reine Angew. Math.}, 512:49--117, 1999.

\bibitem[O'G03]{OG6}
K.~O'Grady.
\newblock A new six-dimensional irreducible symplectic variety.
\newblock {\em J. Algebraic Geom.}, 12(3):435--505, 2003.

\bibitem[Plo07]{Ploog}
D.~Ploog.
\newblock Equivariant autoequivalences for finite group actions.
\newblock {\em Adv. Math.}, 216(1):62--74, 2007.

\bibitem[Pol03]{Polishchukbook}
A.~Polishchuk.
\newblock {\em Abelian varieties, theta functions and the {F}ourier transform},
  volume 153 of {\em Cambridge Tracts in Mathematics}.
\newblock Cambridge University Press, Cambridge, 2003.

\bibitem[SGA73]{SGA}
{\em Th\'{e}orie des topos et cohomologie \'{e}tale des sch\'{e}mas. {T}ome 3}.
\newblock Lecture Notes in Mathematics, Vol. 305. Springer-Verlag, Berlin-New
  York, 1973.
\newblock S\'{e}minaire de G\'{e}om\'{e}trie Alg\'{e}brique du Bois-Marie
  1963--1964 (SGA 4), Dirig\'{e} par M. Artin, A. Grothendieck et J. L.
  Verdier. Avec la collaboration de P. Deligne et B. Saint-Donat.

\bibitem[Ste07]{Stellari}
P.~Stellari.
\newblock Derived categories and {K}ummer varieties.
\newblock {\em Math. Z.}, 256(2):425--441, 2007.

\bibitem[Tar16]{Tari}
K.~Tar\'i.
\newblock {\em Automorphismes des vari\'et\'es de Kummer g\'en\'eralis\'ees}.
\newblock PhD thesis, L'Universit\'e de Poitiers, 2016.

\bibitem[Tat66]{TateThm}
J.~Tate.
\newblock Endomorphisms of abelian varieties over finite fields.
\newblock {\em Invent. Math.}, 2:134--144, 1966.

\bibitem[Yos01]{Yoshioka}
K.~Yoshioka.
\newblock Moduli spaces of stable sheaves on abelian surfaces.
\newblock {\em Math. Ann.}, 321(4):817--884, 2001.

\bibitem[Zar10]{Zarkhin}
Yu.~G. Zarhin.
\newblock Endomorphisms of abelian varieties, cyclotomic extensions, and {L}ie
  algebras.
\newblock {\em Mat. Sb.}, 201(12):93--102, 2010.

\end{thebibliography}
